\documentclass[twoside,12pt,leqno]{amsproc}
\usepackage{amssymb,latexsym,enumerate,tikz,stmaryrd}
\usepackage{longtable,a4wide}
\usepackage[colorlinks,linkcolor=purple,citecolor=blue,urlcolor=blue,hypertexnames=true,pagebackref]{hyperref}
\usepackage{amsrefs}
\usepackage{float}
\numberwithin{table}{section}

\theoremstyle{plain}
\newtheorem{theorem}{Theorem}[section]
\newtheorem*{mthm}{Theorem 1}%[section]
\newtheorem{corollary}[theorem]{Corollary}
\newtheorem{lemma}[theorem]{Lemma}
\theoremstyle{definition}\newtheorem{construction}[theorem]{Construction}
\theoremstyle{definition}
\theoremstyle{definition} %\rm not \it
\theoremstyle{definition}\newtheorem{remark}[theorem]{Remark}   %\rm not \it

\renewcommand{\leq}{\leqslant}
\renewcommand{\ge}{\geqslant}
\renewcommand{\le}{\leqslant}

\newcommand{\Aut}{\mathrm{Aut}}

\newcommand{\C}{\mathcal {C}}
\newcommand{\Char}{\textup{char}}
\newcommand{\E}{\mathbb{E}}

\newcommand{\eps}{\varepsilon}
\newcommand{\F}{\mathbb{F}}
\newcommand{\GAP}{{\sf GAP}}
\newcommand{\GL}{\mathrm{GL}}

\newcommand{\Magma}{\textsc{Magma}}
\newcommand{\OO}{\textup{O}}

\newcommand{\SL}{\mathrm{SL}}
\newcommand{\sym}{\textup{\textsf{S}}}
\newcommand{\Z}{\mathbb{Z}}
\newcommand{\ZZ}{\textup{\textup{Z}}}

\begin{document}

\title[Maximally symmetric $p$-groups]{Maximal linear groups induced on the\\
Frattini quotient of a \scalebox{1.0}{$p$}-group}
\author[J. Bamberg, S.~P. Glasby, L. Morgan, A.~C. Niemeyer]{John Bamberg, S.\,P. Glasby, Luke Morgan and Alice C. Niemeyer}

\address[Bamberg, Glasby\footnote{Also affiliated with The Department of Mathematics, University of Canberra, ACT 2601, Australia.}, Morgan]{Centre for Mathematics of Symmetry and Computation\\
University of Western Australia\\ 35 Stirling Highway\\ Crawley 6009,
Australia.\newline
Email: {\tt John.Bamberg@uwa.edu.au; WWW:
\href{http://www.maths.uwa.edu.au/~bamberg/}{http://www.maths.uwa.edu.au/$\sim$bamberg/}}\newline
Email: {\tt Stephen.Glasby@uwa.edu.au; WWW:
\href{http://www.maths.uwa.edu.au/~glasby/}{http://www.maths.uwa.edu.au/$\sim$glasby/}}
\newline
Email: {\tt Luke.Morgan@uwa.edu.au; WWW: \href{http://www.maths.uwa.edu.au/contact/staff}{http://www.maths.uwa.edu.au/contact/staff}}
}
\address[Niemeyer]{Lehrstuhl B f\"ur Mathematik, Lehr- und Forschungsgebiet Algebra, RWTH Aachen University,
Templergraben 64, 52062 Aachen, Germany.\newline
Email: {\tt Alice.Niemeyer@MathB.RWTH-Aachen.De;\newline
WWW: \href{https://www.mathb.rwth-aachen.de/Mitarbeiter/niemeyer.php}{https://www.mathb.rwth-aachen.de/Mitarbeiter/niemeyer.php}}
}

\date{\today}

\begin{abstract}
Let $p>3$ be a prime. For each maximal subgroup $H\le\GL(d,p)$ with $|H| \geqslant p^{3d+1}$, we
construct a $d$-generator finite $p$-group $G$ with the property that $\Aut(G)$ induces $H$ on the
Frattini quotient $G/\Phi(G)$ and $|G| \leqslant p^{\frac{d^4}{2}}$.  A significant feature of this
construction is that $|G|$ is very small compared to $|H|$, shedding new light upon a celebrated
result of Bryant and Kov\'acs.  The groups $G$ that we exhibit have exponent $p$, and of all such
groups $G$ with the desired action of $H$ on $G/\Phi(G)$, the construction yields groups with
\emph{smallest nilpotency class}, and in most cases, the \emph{smallest order}.
\end{abstract}

\dedicatory{Dedicated to the memory of our distinguished colleague L.G. (Laci) Kov\'acs}

\maketitle

\section{Introduction}

The number of groups of prime power order is dauntingly large: Higman and Sims~\cites{gH,S} showed
that there are as many as $p^{2m^3(1+\OO(m^{-1/3}))/27}$ groups of order $p^m$.  This suggests that
properties of $p$-groups should be investigated statistically. Given a property of $p$-groups, one
may ask: What is the \emph{range} of possibilities? What is the \emph{frequency} distribution? What
are the \emph{mean} and \emph{variance}?

Some questions concerning `ranges' were considered in the 1970's. For example, one may ask which
groups can arise as the group induced by the automorphism group $\Aut(G)$ acting on $G/\ZZ(G)$, for
a $p$-group $G$.  Heineken and Liebeck~\cite{HL} showed that the range is as large as possible,
namely for any finite group $H$ and any prime $p>2$, there exists a $p$-group~$G$ of nilpotency
class~2, and exponent $p^2$ such that $\Aut(G)$ induces $H$ on ${G/\ZZ(G)}.$ Later this result was
generalised to $p=2$, see~\cites{aH, Webb}.  The group $G$ constructed in~\cite{HL} is a
$d$-generator $p$-group where $d=|H|\binom{k+2}{2}$ and $H$ is $k$-generated.  Soules and
Woldar~\cite{SW} reduce the number of generators of~$G$ to $d=|H|$ when $H$ is a sporadic simple
group. These examples have $|G|>p^{|H|}$ so it is unclear whether one sees such wild behaviour in
practical examples, or whether $|G|$ is always huge compared to $|H|$. Is wildness of theoretical
interest only?

A result addressing the \emph{frequency} is due to Helleloid and Martin. They show in \cite[Theorem
  3]{HM} that the group $A(G)$ induced on $G/\Phi(G)$ by the automorphism group of some
$d$-generator $p$-group~$G$, is `almost always' the trivial subgroup of $\GL(d,p)$.  In light of
this result, a natural question about \emph{ranges} is: Which subgroups $H\le\GL(d,p)$ are
conjugate\footnote{After a basis has been chosen for $G/\Phi(G)$, we may regard $A(G)$ as a subgroup
  of $\GL(d,p)$, we thus speak of conjugacy to mean `up to change of basis'. We will write $A(G)=H$
  to mean a basis may be chosen to effect this equality.}  to $A(G)$, for some $d$-generator
$p$-group~$G$?  Thus groups for which $A(G)$ is non-trivial are rare.  However, Bryant and
Kov\'acs~\cite{BK} prove a striking result: given \emph{any} $H\le\GL(d,p)$ where $d>1$, there
exists a $d$-generator $p$-group~$G$ such that $\Aut(G)$ induces on $G/\Phi(G)$ the linear
group~$H$.  An alternative proof of this celebrated result is given in \cite[Chapter~VIII,
  \S13]{HB2}.  Whilst the methods of the proof of~\cite[Theorem~1]{BK} are natural, utilising the
Lie ring associated to a $p$-group, the conclusion is not constructive: it is an existence result
bounding neither $|G|$, nor the nilpotency class of~$G$, nor the exponent of~$G$.

Inspired by the above results, given $H\le\GL(d,p)$, we ask: Is it possible to find relatively small
groups~$G$ (compared to $|H|$) satisfying $A(G)=H$?  For certain classes of $H$, of order at least
$p^{3d+1}$, we construct a $d$-generator finite $p$-group $G$ with the property that $A(G) = H$ and
$|G| \leqslant p^{ \frac{d^4}{2} }$.  Thus, our construction shows that `small' $p$-groups~$G$ with
$A(G)=H$ do in fact occur.  Our methods for constructing~$G$ from $H$ involve representation theory;
our constructions are geometric, and we believe, also very natural. We hope that they contribute to
a deeper understanding of automorphism groups of $p$-groups and their construction, as even the very
efficient algorithms \cite{ELOB} to compute $\Aut(G)$ struggle when $G$ is large, for example, when
$G$ is one of the groups we construct in Table ~\ref{Tab2}.  For more information on automorphism
groups of $p$-groups, we refer the reader to the survey of Helleloid \cite{H}.

To state our main result we require the following definition.  The \emph{lower exponent-$p$ central
  series}\footnote{Properties of this series are given in Huppert and Blackburn~\cite[16, Chapter
    VIII]{HB3}. However, their definition differs from ours as it starts with $X_1=X$.} for a
group~$X$ is defined inductively by $X_0=X$, $X_k=[X,X_{k-1}]X_{k-1}^p$ for $k\ge1$.  The smallest
integer $n$ for which $X_n=\{1\}$ (when it exists) is called the \emph{lower $p$-length} of $X$, and
we write~$n_p(X)=n$.  If $X$ is a group of exponent $p$, the lower $p$-length of $X$ is equal to the
\emph{nilpotency class} of $X$ (or \emph{class} for short).  With our numbering convention
($X_0=X$), we have $[X_i,X_j]\le X_{i+j+1}$ for all $i,j\ge0$. We alert the reader that for some
authors $X_i$ denotes the $(i+1)$st term of the lower central series for $X$.

\begin{mthm} 
\label{T:main}
Let $p >3$ be a prime, and let $d>1$ be an integer. Suppose that $H$ is a maximal subgroup of
$\GL(d,p)$ with $\SL(d,p)\not\le H$ and that $|H| \geqslant p^{3d+1}$.  Then there exists a
$d$-generator $p$-group~$G$ of exponent~$p$, class at most~$4$, order at most $p^{\frac{d^4}{2}}$
and such that $\Aut(G)$ induces $H$ on the Frattini quotient $G/\Phi(G)$.  The nilpotency class,
order and structure of~$G$ is given in Table~{\rm\ref{Tab2}}.
\end{mthm}

\subsection{Strategy and outline of the paper}

We address the problem: given $H\le\GL(d,p)$ find~$G$ such that $A(G)=H$. To ensure that~$G$ is
interesting, we choose $H$ to be a \emph{maximal} subgroup of $\GL(d,p)$, and insist that $|G|$ is
\emph{minimal} subject to having exponent~$p$ (c.f.~Remark~\ref{rem:category}). To avoid
trivialities, we assume that $p>2$ (as 2-groups of exponent 2 are elementary abelian). In
Section~\ref{S4} we summarise the maximal subgroups of $\GL(d,p)$ that we consider and explain the
notation in Columns 1-4 of Table~\ref{Tab2}.

Our strategy for constructing $G$ is to examine the freest $d$-generator group $B$ of exponent
$p$. The quotient $\Gamma(d,p,n) = B/B_n$ (the quotient of $B$ by the $n$th term of its lower
central series) is the universal $d$-generator $p$-group of exponent $p$ and class $n$.  Our results
depend critically on a practical description of $\Gamma(d,p,n)$.  In \S\ref{S2} we
describe~$\Gamma(d,p,n)$ using a new data structure which we call \emph{Lie $n$-tuples}.  The
problem of constructing our desired group~$G$ is reduced in \S\ref{S2} to determining the
$H$-submodule structure of a certain Lie power $ L^nV$ of the natural $H$-module $V$, see
Theorem~\ref{T:trick}. In \S\ref{S3} we consider the irreducible submodules of Lie powers, keeping
the prerequisites to a minimum.  Aschbacher's classes $\C_i$ of maximal subgroups $H$ of $\GL(d,p)$
are listed in \S\ref{S4} before we determine class-by-class the $H$-submodule structure of $L^nV$ in
\S\ref{S5}. The proof of Theorem~\ref{T:main} appears in \S\ref{S7}, and we conclude in \S\ref{S8}
with some open questions and directions for future research.

\subsection*{Notation}

Throughout the paper, $V$ will denote a vector space of dimension $d$ over a (possibly infinite)
field $\F$. The precedence of the operators\footnote{The $n$th alternating, symmetric and tensor
  powers of $V$ are denoted $A^nV$, $S^nV$ and $T^nV$, respectively.} $A^n, S^n, T^n$ is greater
than $\otimes$ which is greater than $\oplus$.  For example, $A^nU \otimes V\oplus W$ means $((A^nU)
\otimes V)\oplus W$.

\section{Universal groups of exponent \texorpdfstring{$p$}{p}}\label{S2}

We fix integers $d$ and $n$ and a prime $p$. In this section, we discuss the universal group in the
category of finite $d$-generator $p$-groups of class $n$ and exponent $p$. First, we approach this
group from an abstract point of view, and later realise this group concretely. We set the following
notation:
\begin{itemize}
\item $F(d)$, the free group of rank $d$,
\item $B(d,p) = F(d)/F(d)^p$, the relatively free group of rank $d$ and exponent $p$,
\item $\Gamma(d,p,n) = B(d,p) / B(d,p)_n$, the relatively free group of rank $d$, exponent $p$ and class $n$.
\end{itemize}

Note that the group $\Gamma(d,p,n)$ is finite, having bounded rank, exponent and class. Moreover,
$\Gamma(d,p,n)$ is universal, in the sense that each finite $p$-group of rank $d$, exponent $p$ and
class $n$ is an image of $\Gamma(d,p,n)$. An explicit formula for the order for $\Gamma(d,p,n)$ was
given by Witt; to describe this formula we require some additional knowledge of Lie rings.
 
Higman describes in \cite{Higman} how to associate a graded Lie ring $L_{(N_i)}$ to a normal series
$G=N_1\trianglerighteqslant N_2 \trianglerighteqslant \cdots$ for a group $G$ provided $[N_i,N_j]
\le N_{i+j}$ and $\bigcap_{i=1}^\infty N_i=\{1\}$ hold.  The $N_i/N_{i+1}$ are abelian as $[N_i,N_i]
\le N_{2i}\le N_{i+1}$. We view the $N_i/N_{i+1}$ as \emph{additive} groups, and then form the
abelian group $L_{(N_i)}=\bigoplus_{i=1}^\infty N_i/N_{i+1}$. The following multiplication rule
$(g_i N_i)(g_j N_j) = [g_i,g_j] N_{i+j}$ turns $L_{(N_i)}$ into a graded Lie ring. The Hall-Witt
identity for $G$ (see~\cite{Tao}) gives rise to the Jacobi identity for $L_{(N_i)}$. The sections
$N_i/N_{i+1}$ are called \emph{homogeneous components} of the Lie ring.

Returning now to $F(d)$ and $B(d,p)$, both the lower central series\footnote{The \emph{lower central
    series} of a group $X$ is defined by $\gamma_1(X):=X$ and $\gamma_{i+1}(X):=[\gamma_i(X),X]$
  for~$i\ge1$.} of~$F(d)$ (taking $N_i=\gamma_i(F(d))$) and the lower exponent-$p$ central series of
$B(d,p)$ (taking $N_i=B(d,p)_{i-1}$) satisfy the conditions $[N_i,N_j] \le N_{i+j}$ and
$\bigcap_{i=1}^\infty N_i=\{1\}$. This gives two related Lie rings which we denote simply by
${\mathcal L}$ and $L$:
\begin{equation}\label{E:LL}
  \mathcal L:=L_{(\gamma_i(F))} =\bigoplus_{i=1}^\infty   {\mathcal L}^i             \quad               \textup{and}\quad L:=L_{(B_{i-1})}  =\bigoplus_{i=1}^\infty  L^i,
\end{equation}
where  ${\mathcal L}^k$ and $L^k$ are the $k$th homogeneous components of $\mathcal L$ and $L$,  respectively.

It turns out that ${\mathcal L}^k$ is a free abelian group, and $L^k$ is a vector space over the
prime field~$\F_p$. Witt \cite[Satz~3]{W2} gave formulas for the rank $f(d,k)$ of ${\mathcal L}^k$,
and dimension $f_p(d,k)$ of $L^k$. Indeed,
\begin{equation}\label{E1}
  {\mathcal L}^k\cong\Z^{f(d,k)}\quad\textup{where}\quad
  f(d,k)=\frac1k\sum_{i\mid k}\mu(i)d^{\frac ki},
\end{equation}
and $\mu$ is the number theoretic M\"obius function.
Also, by \cite[p.209\;(6p)]{W}, we have
\begin{equation}\label{E1p}
  L^k \cong (\F_p)^{f_p(d,k)}\quad\textup{where}\quad
  f_p(d,k)=\frac1k\sum_{i\mid k}\mu(i_0)\varphi(p^h)d^{\frac ki}
  \quad(i=i_0p^h, p\nmid i_0),
\end{equation}
and $\varphi$ is Euler's totient function.  Note that $f(d,k)=f_p(d,k)$ if $p > k$, and $F_{k-1}/F_k
= \bigoplus_{i=1}^k L^i$ by~\cite[Theorem~16]{HM}.  This is illustrated in Figure~\ref{F1}.

\begin{figure}[ht!]
  \caption{The lower central series of $F:=F(d)$ and the lower exponent-$p$ central series for $F$
    and $B:=B(d,p)$. The Lie algebras ${\mathcal L}$ and $L$ have the sections in the first and
    third chains.}  \vskip2mm
  \label{F1}
\begin{tikzpicture}[xscale=0.6, yscale=0.6]
  \draw[fill] (0,-1) circle [radius=0.04];
  \draw[fill] (0,-0.7) circle [radius=0.04];
  \draw[fill] (0,-0.4) circle [radius=0.04];
  \draw[fill] (0,0) circle [radius=0.1];
  \draw[fill] (0,1) circle [radius=0.1];
  \draw[fill] (0,2) circle [radius=0.1];
  \draw[fill] (0,3) circle [radius=0.1];
  \draw[line width=1pt,fill] (0,0) node[left,black] {$\gamma_4(F)$} -- (0,1);
  \draw[line width=1pt,fill] (0,1) node[left,black] {$\gamma_3(F)$} -- (0,2);
  \draw[line width=1pt,fill] (0,2) node[left,black] {$\gamma_2(F)$} --  (0,3) node[left,black] {$F=\gamma_1(F)$};
  \draw[thick] (0,2.5) node[right] {$\;\mathcal L^1=\Z^d $};
  \draw[thick] (0,1.5) node[right] {$\;\mathcal L^2=\Z^{(d^2-d)/2}$};
  \draw[thick] (0,0.5) node[right] {$\;\mathcal L^3=\Z^{(d^3-d)/3}$};
\end{tikzpicture}
\hspace{10mm}
\begin{tikzpicture}[xscale=0.6, yscale=0.6]
  \draw[fill] (0,-1) circle [radius=0.04];
  \draw[fill] (0,-0.7) circle [radius=0.04];
  \draw[fill] (0,-0.4) circle [radius=0.04];
  \draw[fill] (0,0) circle [radius=0.1];
  \draw[fill] (0,1) circle [radius=0.1];
  \draw[fill] (0,2) circle [radius=0.1];
  \draw[fill] (0,3) circle [radius=0.1];
  \draw[line width=1pt,fill] (0,0) node[left,black] {$F_3$} -- (0,1);
  \draw[line width=1pt,fill] (0,1) node[left,black] {$F_2$} -- (0,2);
  \draw[line width=1pt,fill] (0,2) node[left,black] {$F_1$} --  (0,3) node[left,black] {$F=F_0$};
  \draw[thick] (0,2.5) node[right] {$\;L^1$};
  \draw[thick] (0,1.5) node[right] {$\;L^1\oplus L^2$};
  \draw[thick] (0,0.5) node[right] {$\;L^1\oplus L^2\oplus L^3$};
\end{tikzpicture}
\hspace{10mm}
\begin{tikzpicture}[xscale=0.6, yscale=0.6]
  \draw[fill] (0,-1) circle [radius=0.04];
  \draw[fill] (0,-0.7) circle [radius=0.04];
  \draw[fill] (0,-0.4) circle [radius=0.04];
  \draw[fill] (0,0) circle [radius=0.1];
  \draw[fill] (0,1) circle [radius=0.1];
  \draw[fill] (0,2) circle [radius=0.1];
  \draw[fill] (0,3) circle [radius=0.1];
  \draw[line width=1pt,fill] (0,0) node[left,black] {$B_3$} -- (0,1);
  \draw[line width=1pt,fill] (0,1) node[left,black] {$B_2$} -- (0,2);
  \draw[line width=1pt,fill] (0,2) node[left,black] {$B_1$} --  (0,3) node[left,black] {$B=B_0$};
  \draw[thick] (0,2.5) node[right] {$\;L^1$};
  \draw[thick] (0,1.5) node[right] {$\;L^2$};
  \draw[thick] (0,0.5) node[right] {$\;L^3$};
\end{tikzpicture}
\end{figure}

Below we summarise the above discussion.

\begin{lemma}
\label{L3}
Suppose that $p>n$ and $1 \leq k\le n$. Then we have
\[
  | \gamma_k(\Gamma(d,p,n))/\gamma_{k+1}(\Gamma(d,p,n))| = p^{f(d,k)}
  \quad\textup{where $f(d,k)= \frac{1}{k}\sum_{i\mid k}\mu(i)d^ {\frac{k}{i}}$}.
\]
\end{lemma}

Next we turn to the automorphism group of $\Gamma(d,p,n)$, and of certain quotients.

\begin{theorem}
\label{T:trick}
Let $B=B(d,p)$.  If $B_n\le M < B_{n-1}$ and $G=B/M$, then $A(G)=K$ where $K=\mathrm
N_{\GL(d,p)}(M/B_n)$, i.e., the group $A(G)$ of automorphisms induced by $\Aut(G)$ on $G/\Phi(G)$ is
$K$.  Furthermore, the nilpotency class of $G$ is $n$.
\end{theorem}

\begin{proof}
First, $B_n\le M < B_{n-1}$ implies $n_p(G)=n$ as $G_{n-1}=B_{n-1}/M$ is non-trivial.  Second, the
proof relies on the fact that $A(B) \cong \GL(d,p)$ induces a well-defined action on the elementary
abelian $p$-groups $B_{n-1}/B_n$, see~\cite[Chapter VIII, Lemma 13.3 and Theorem 13.4]{HB2}
and~\cite[\S2.2]{HM}.  For the remainder of the proof, see~\cite[Theorem~13]{HM}.
\end{proof}

In order to apply Theorem~\ref{T:trick}, it is useful to have a more explicit description of
$\Gamma(d,p,n)$. Construction~\ref{const:lie tuple} below achieves this and it relates the action of
automorphisms to linear actions in an explicit way.  Let $V=\F^d$ be a $d$-dimensional module over a
field of characteristic $p$. View $V$ as a $\GL(V)$-module, and consider the tensor algebra
$T(V)=\bigoplus_{n\ge0}T^nV$ where each $T^nV=V^{\otimes n}$ is a $\GL(V)$-module. For $u,v\in T(V)$
define
\begin{equation}
\label{wedge}
  [u,v]:=u\otimes v-v\otimes u,
\end{equation}
and let $L(V)$ be the closure of $V$  under this bracket operation. Then
$L(V)=\bigoplus_{n\ge1}L^nV$ is a free Lie $\F$-algebra by Witt's Theorem,
where $L^nV:=T^nV\cap L(V)$ is called the $n$-th \emph{Lie power} of $V$,
see~\cites{BK,J}.
Note that $L^1 V =V=T^1V$ and $[L^i V ,L^j V ]\subseteq L^{i+j} V $ for $i,j\geqslant 1$.

\begin{construction}[Lie $n$-tuples]
\label{const:lie tuple}
Let $V$ be a $d$-dimensional vector space over a field $\F$ of characteristic $p$, and assume that
$p>n$.  We set $$\Gamma_n(V):=\prod_{i=1}^n L^i V .$$ We write typical elements of $\Gamma_n(V)$ as
$g_n=(v_1,\dots,v_n)$, $g'_n=(v'_1,\dots,v'_n)$ and $g''_n=(v''_1,\dots,v''_n)$ where
$v_i,v'_i,v''_i\in L^i V $.  A binary operation $g_ng'_n=g''_n$ on $\Gamma_n(V)$ is a rule for
writing the $v''_k$ in terms of the $v'_j$ and $v_i$.

The operation for $\Gamma_1(V)=V$ is addition.
For $n=2, 3, 4$ it is defined as follows:
\begin{align}
  g_2g'_2&=(v_1+v'_1,v_2+v'_2+[v_1,v'_1]),\label{E2.1}\\
  g_3g'_3&=(v_1+v'_1,v_2+v'_2+[v_1,v'_1],v_3+v'_3
    +3[v_1, v_2'] + 3[v_2, v_1'] +[v_1, v_1', v_1' - v_1]),\label{E2.2}\\
  g_4g'_4&=(v_1+v'_1,v_2+v'_2+[v_1,v'_1],v_3+v'_3
    +3[v_1, v_2'] + 3[v_2, v_1'] +[v_1, v_1', v_1' - v_1],\label{E2.3}\\
        &\hspace{7mm}v_4+v_4'+[v_1, v_3'] +3[v_2,v_2']+[v_3, v_1'] 
        \nonumber\\
  &\hspace{11mm} +[v_2, v_1',v_1'-v_1] + [v_1, v_2',v_1'-v_1]+[v_1,v_1',v_2'-v_2] -  [v_1, v_1', v_1, v_1']).\nonumber
\end{align}
where for notational convenience, left-normed Lie brackets such as
$[[[v,v'],v''],v''']$ are abbreviated by $[v,v',v'',v''']$. \hfill $\lozenge$
\end{construction}

\begin{remark}\label{R:BCH}
When $n< p$, the Lazard correspondence applied to the finite nilpotent Lie ring $L(V)/
\bigoplus_{i>n}^\infty L^i V$ of class~$n$ gives a group of the same order and class which turns out
to be isomorphic to our $p$-group $\Gamma_n(V)$ when $n\le 4$. This observation allows us to deduce
a multiplication rule for $\Gamma_n(V)$ for $n>4$ from the Baker-Campbell-Hausdorff formula (see
~\cite{CdGVL} for a nice overview). The rules~\eqref{E2.1}--\eqref{E2.3} above allow us to do
practical computations with the automorphism group of $\Gamma_n(V)$, as will become apparent
below. For example, we identify the Lie elements $x=x_1+\frac12 x_2+\frac1{12}x_3+\frac1{24}x_4$ and
$y=y_1+\frac12 y_2+\frac1{12}y_3+\frac1{24}y_4$ with the group elements $(x_1,x_2,x_3,x_4)$ and
$(y_1,y_2,y_3,y_4)$ where $x_i,y_i\in L^i(V)$ and then we substitute $x,y$ into the left-normed BCH
formula for $z(x,y)$ where $e^xe^y=e^{z(x,y)}$ (\emph{c.f.} \cite[p.\,432]{CdGVL}):
\[
  z(x,y)=x+y+\frac12[x,y]-\frac1{12}[x,y,x]+\frac1{12}[x,y,y]-\frac1{24}[x,y,x,y]+\cdots.
\]
Expressing the answer in the form $z=z_1+\frac12 z_2+\frac1{12}z_3+\frac1{24}z_4$ by expanding
modulo $\bigoplus_{i>4}^\infty L^i V$ gives the rule~\eqref{E2.3}.
\end{remark}

\begin{theorem}\label{T1}
Let $V=\F^d$ be a $d$-dimensional space over $\F$. Then
\begin{enumerate}[{\rm (i)}]
  \item $\Gamma_2(V)$ is a group of order
    $|\F|^{d(d+1)/2}$, and class~$2$ when $\Char(\F)\ne2$.
  \item $\Gamma_3(V)$ is a group of order
    $|\F|^{d(d+1)(2d+1)/6}$, and class~$3$ when $\Char(\F)\ne2,3$.
  \item $\Gamma_4(V)$ is a group of order
    $|\F|^{d(d+1)(3d^2+d+2)/12}$, and class~$4$ when $\Char(\F)\ne2,3$.
  \item If $|\F|=p$, $p>n$ and $n\le4$, then
  $$ \Gamma(d,p,n) \cong \Gamma_n(\F_p^d).$$ In particular, $\Gamma_n(\F_p^d)$ has exponent~$p$ and class~$n$.
  \item For $n\leqslant 4$ there is a monomorphism $\alpha : \GL(V) \rightarrow \Aut(\Gamma_n(V))$ defined by $\alpha : g \mapsto \alpha_g$, where $\alpha_g$ is as follows:
$$(v_1,\dots,v_n)\alpha_g=(v_1g,\dots,v_ng).$$
  \item Suppose $p>n$, $n\le4$, and $V=\F_p^d$, then
  \[
     \Aut(\Gamma_n(V))=K \rtimes\GL(V),
  \]
  where $K$ is the kernel of the action of $\Aut(\Gamma_n(V))$ on the
  quotient $\Gamma_n(V)/\Phi(\Gamma_n(V))$.
\end{enumerate}
\end{theorem}

\begin{proof}
(i)--(iii) The associative law $(g_ng'_n)g''_n=g_n(g'_ng''_n)$ follows from the Lazard
  correspondence when $\Char(\F)>n$.  It is noteworthy that associativity holds even when
  $\Char(F)\le n$. It holds for $n=1$ because $(v_1+v'_1)+v''_1=v_1+(v'_1+v''_1)$, and it holds for
  $n=2$ because $[\;,\;]$ is biadditive. Verifying associativity for $n=3, 4$ involves complicated
  (though technically simple) calculations. For this reason we delegated the task to a
  \Magma~\cite{BCP} computer program whose source can be found at~\cite{G}.  The identity element is
  easily seen to be the all zeroes vector, written $1=(0,\dots,0)$, and the inverse of $g_n$ is
  $g_n^{-1}=(-v_1,\dots,-v_n)$. This follows because $[v,0]=[0,v]=[v,-v]=0$. Hence $\Gamma_n(V)$ is
  a group for $n\le4$ and all vector spaces $V=\F^d$.

Properties of these groups depend on the characteristic of the field $\F$.  For example, it is easy
to see by induction on $k$ that $g_n^k=(kv_1,\dots,kv_n)$ for $k\in\Z$. Hence $\Gamma_n(V)$ has
exponent~$p$ if $\Char(\F)=p>0$, and is torsion-free otherwise. The following commutator
calculations are too long for most humans (when $n=3, 4$) and were done by the \Magma~\cite{BCP}
computer programs in~\cite{G}:
\begin{align}
  [g_2,g'_2]&=(0,2[v_1,v'_1]), \label{E3.1}\\
  [g_3,g'_3,g''_3]&=(0,0,12[v_1,v'_1,v''_1]),\label{E3.2}\\
  [g_4,g'_4,g''_4,g'''_4]&=(0,0,0,24[v_1,v'_1,v''_1,v'''_1]),\label{E3.3}
\end{align}
where for notational convenience, left-normed group commutators such as $[[[g,g'],g''],g''']$ are
abbreviated by $[g,g',g'',g''']$.

The order of $\Gamma_n(V)$ is $\prod_{i=1}^n|L^i V |$ and $|L^i V |=|\F|^{f(d,i)}$ by~\cite{W} where
$f(d,i)$ is given by~\eqref{E1}. Moreover, it follows from \eqref{E3.1}, \eqref{E3.2}, \eqref{E3.3}
that $\Gamma_n(V)$ has class~$n$ if $\Char(\F)\not\in\{2,\dots,n\}$. This proves parts (i)--(iii).

(iv) Suppose now that $\F=\F_p$, and consider part~(iv) for $n\le4$.  As $p>n$, Lemma~\ref{L3} shows
that $n_p(\Gamma(d,p,n))= n$ and $|\Gamma(d,p,n)|=\prod_{i=1}^n|L^i V |$. Thus it follows that
$\Gamma(d,p,n)\cong \Gamma_n( \F_p ^d)$, as desired.

(v) Each $g\in\GL(V)$ induces an action on $L^nV$.  A significant advantage of the
definitions~\eqref{E2.1}, \eqref{E2.2}, \eqref{E2.3} is that the map $\alpha_g$ is easily verified
to be an endomorphism of $\Gamma_n(V)$. In fact, $\alpha_g$ is an automorphism with inverse
$\alpha_{g^{-1}}$. Thus the map $\alpha\colon\GL(V)\to\Aut(\Gamma_n(V))$ with $\alpha(g)=\alpha_g$
is a monomorphism.

(vi) The action of $\Aut(\Gamma_n(V))$ on the Frattini quotient $\Gamma_n(V)/\Phi(\Gamma_n(V))\cong
V$ induces a homomorphism $\Aut(\Gamma_n(V))\to\GL(V)$, which is surjective by part (v). We have now
shown that $\GL(V)$ is a subgroup (and a quotient group) of $\Aut(\Gamma_n(V))$. Hence
$\Aut(\Gamma_n(V))$ splits as $\Aut(\Gamma_n(V))=K \rtimes\GL(d,p)$ for $n\le4$, with $K$ as in the
statement above. In fact, $K$ is a normal $p$-subgroup of $\Aut(\Gamma_n)$ by a theorem of Hall.
\end{proof}

\begin{remark}
  The constants appearing in the commutator relations given in \eqref{E3.1}, \eqref{E3.2} and
  \eqref{E3.3} are denominators appearing in the Baker-Campell-Hausdorff formula.  The connection is
  related to the Lazard correspondence as explained in Remark~\ref{R:BCH}.
\end{remark}

\begin{remark}\label{R}
One may guess that rules~\eqref{E2.1}--\eqref{E2.3} for multiplying Lie $n$-tuples do no more than
encode a pc-presentation\footnote{The abbreviation `pc' stands for `power-conjugate',
  `power-commutator' or `polycyclic', see \cite{EHOB}.} for $\Gamma_n(V)$.  This turns out
\emph{not} to be the case.  For example, consider a special group~$\Gamma_2(V)=G$ of
order~$p^{\binom{m}{1}+\binom{m}{2}}$ and exponent~$p>2$ where $V=(\F_p)^m$. Let $G$ have generators
$g_i$, $1\le i\le m$, and $h_{k,j}$, $1\le j<k\le m$, and define a pc-presentation for $G$ by
$g_i^p=h_{k,j}^p=1$, and $g_j^{g_k}=g_jh_{k,j}$ for $1\le j<k\le m$.  This pc-presentation gives
rise to the symbolic multiplication rule
\begin{equation}\label{E:pcsymb}
  \left(\prod_{i=1}^mg_i^{x_i}\prod_{j<k}h_{k,j}^{y_{k,j}}\right)
  \left(\prod_{i=1}^mg_i^{x'_i}\prod_{j<k}h_{k,j}^{y'_{k,j}}\right)=
  \prod_{i=1}^mg_i^{x_i+x'_i}\prod_{j<k}h_{k,j}^{y_{k,j}+y'_{k,j}+x_jx'_k}.
\end{equation}
Indeed when $m=1$, \emph{every} pc-presentation for~$G$ (with different composition series or
transversals) has the same rule.  It is much easier to prove that $\GL(V)$ is a subgroup of
$\Aut(G)$ using the more geometric `Lie' rule~\eqref{E2.1}, than using~\eqref{E:pcsymb}.  We return
to this point in Remark~\ref{R:extra}.
\end{remark}

\section{Some representation theory}\label{S3}

Bryant and Kov\'acs proved \cite[Theorem 1]{BK} by considering regular submodules of a certain sum
of Lie powers \cite[Theorem 2]{BK}. In this section, we consider the relevant Lie representation
theory for our results. A good introduction to this topic is~\cite{J}. As noted in \S\ref{S2}, the
action of $\GL(V)$ on $V$ induces an action on the tensor algebra $T(V)$, and on $L(V)$ (which is a
subset of $T^nV$ containing $V$, closed under the Lie bracket $[\;,\;]$).

Our aim in this section is to describe the $\GL(V)$-modules $L^i V $ for $1\leqslant i \leqslant 4$
and to show that they are irreducible. We note that the representation theory of $\GL(V)$ on $T^nV$
is known when $\Char(\F)=0$ (see \cite{FH}) and the irreducible $\GL(V)$-modules are described by
the representation theory of the symmetric group $\sym_n$ of degree $n$. We require the analogous
results when $\F$ is a finite field and $\Char(\F) > n$, which we have been unable to locate in the
literature.

The action of $g\in\GL(V)$ on the $n$th tensor power $T^nV=V^{\otimes n}$ is
\[
  (v_1\otimes\cdots\otimes v_n)g=(v_1g)\otimes\cdots\otimes (v_ng)
  \qquad\textup{where $v_1,\dots,v_n\in V$,}
\]
and  the following action of the symmetric group of degree $n$   commutes with that of $\GL(V)$:
\[
  (v_1\otimes\cdots\otimes v_n)\sigma=(v_{1\sigma^{-1}})\otimes\cdots\otimes (v_{n\sigma^{-1}})
  \qquad\textup{where $v_1,\dots,v_n\in V$, and $\sigma \in \sym_n$.}
\]

Suppose now that $\Char(\F)\not\in\{2,\dots,n\}$ so that $\sym_n$ acts completely reducibly on
$T^nV$. There exist primitive central orthogonal idempotents\footnote{This means $\sum_{i=1}^r
  e_i=1$, $e_i^2=e_i \in \ZZ(\F \sym_n)$ for $1\le i\le r$, and $e_ie_j=0$ for $1\le i<j\le r$.}
$e_1,\dots,e_r \in \F \sym_n$ which satisfy
\begin{equation} \label{decomp by sn}
T^nV = \bigoplus_{i=1}^r (T^nV)e_i.
\end{equation}
Since the actions of $\GL(V)$ and $\sym_n$ commute, this is a
$\GL(V)$-invariant decomposition of~$T^nV$. 
The primitive idempotents
\[
  e_1 = \frac{1}{n!}\left(\sum_{\sigma \in \sym_n} \sigma\right)\quad\textup{and}\quad
  e_2 = \frac{1}{n!}\left(\sum_{\sigma \in \sym_n} \mathrm{sgn}(\sigma)\sigma\right)
\]
give rise to the  \emph{symmetric} and \emph{alternating} powers
$S^nV$ and $A^nV$, respectively. For vectors $v_1, \dots, v_n \in V$ we define
$$v_1 \odot \dots \odot v_n = n! (v_1 \otimes \dots \otimes v_n)e_1 \quad \text{and} \quad v_1
\wedge \dots \wedge v_n = n!  (v_1 \otimes \dots \otimes v_n)e_2 .$$ The symmetric and alternating
powers are spanned by vectors of the form $v_1 \odot \dots \odot v_n$ and $v_1 \wedge \dots \wedge
v_n$ respectively, and their dimensions are
\begin{equation}\label{E:AS}
  \dim(S^nV)=\binom{d+n-1}{n}\quad\textup{and}\quad
  \dim(A^nV)=\binom{d}{n}.
\end{equation}

For the case $n=2$, this gives
\begin{equation}\label{E:A+S}
  T^2V=V\otimes V=A^2V\oplus S^2V\qquad\textup{if $\Char(\F)\ne2$.}			
\end{equation}

We now relate $L^nV$ for $n\le3$, to more familiar modules.  We have $L^1 V =V$ and $L^2 V =A^2V$
because $ v_1\wedge v_2= [v_1,v_2]$ (see \eqref{wedge}).  We warn the reader that $[v_1,v_2,v_3]\ne
v_1\wedge v_2\wedge v_3$; the left hand side term has four summands while the right hand side term
has six summands.

\begin{lemma}
\label{L4}
Suppose that   $\Char(\F)\ne2,3$. The following hold.
\begin{enumerate}[{\rm (i)}]
  \item If $d\ge3$, then $L^3 V =    X_1\oplus X_2$ is a sum of irreducible
  $H$-modules, where $H$ is the group $\GL(1,\F)\wr\sym_d$ of all
  monomial matrices, and $\dim(X_1)=2\binom{d}{2}$ and
  $\dim(X_2)=2\binom{d}{3}$.
  \item If $d>1$, then $L^3V$ is an irreducible $\GL(V)$-module.
  \item There are isomorphisms
  $A^2V \otimes V\cong L^3V \oplus A^3V$, and	
  $S^2V \otimes V\cong S^3V\oplus L^3 V $ of $\GL(V)$-modules. Hence
  $T^3V\cong S^3V\oplus L^3V\oplus L^3V \oplus A^3V$.
\end{enumerate}
\end{lemma}

\begin{proof}
(i)~Suppose that $H$ preserves a decomposition $V=V_1\oplus\cdots\oplus V_d$ where the 1-dimensional
  subspaces $V_i = \langle v_i \rangle$ are permuted transitively.  Let $K:=G_1\times\cdots\times
  G_r$ be the base group of $H=\GL(V_1)\wr\sym_d$ where $G_i=\GL(V_i)$.  For $i,j,k$ there are three
  possibilities for the dimension of $V_i + V_j + V_k$, depending on the cardinality of the
  set~$\{i,j,k\}$.  For $A,B \subseteq V$ let $[A,B]:=\langle [a,b]\mid a\in A, b\in B\rangle$. Then
  $A^2V\otimes V$ has two obvious $H$-submodules:
\[
  W_1 =\sum_{i<j}\; [V_i,  V_j]\otimes (V_i+V_j),
  \qquad\textup{and}\qquad W_2
  =\sum_{k\not\in\{i,j\}} [V_i,V_j]\otimes V_k.
\]
It is clear that $A^2V\otimes V= W_1 + W_2$.  Since 
\[
  \dim(W_1)+\dim(W_2)\le 2\binom{d}{2}+ d\binom{d-1}{2}
  =d\binom{d}{2}=\dim(A^2V\otimes V),
\]
the inequality above is an equality and $A^2V\otimes V=W_1 \oplus W_2$ is an
$H$-module decomposition.  Now
\begin{equation}\label{EA3V}
  v_1\wedge v_2\wedge v_3
  =[v_1,v_2]\otimes v_3+[v_2,v_3]\otimes v_1+[v_3,v_1]\otimes v_2,
\end{equation}
so $W_2$ contains $A^3V$ and $\dim(W_2/A^3V)=d\binom{d-1}{2}-\binom{d}{3}=2\binom{d}{3}>0$ since
$d\ge3$.  We claim that $W_1$ and $W_2 / A^3 V$ are irreducible $H$-modules.

We may write each 2-dimensional subspace $[V_i,V_j]\otimes(V_i +V_j)$ of $W_1$ as the sum of two
1-dimensional $K$-invariant subspaces, which are isomorphic to $[V_i,V_j]\otimes V_i$ and
$[V_i,V_j]\otimes V_j$ respectively. Hence $W_1$ can be written as the sum of $2\binom{d}{2}$
1-dimensional subspaces that are pairwise non-isomorphic as $K$-modules. As these are permuted
transitively by $H$, we find that $W_1$ is an irreducible $H$-module.

For $W_2$, let $\Delta$ be the set of $3$-subsets of $\{1,\dots,d\}$.  For each $\delta=\{i,j,k\}$
in $\Delta$ define
\[
  U_\delta:= [V_i,V_j]\otimes V_k + [V_j,V_k]\otimes V_i + [V_k,V_i]\otimes V_j.
\]
Then $W_2 = \bigoplus_{\delta \in \Delta} U_\delta$.  The diagonal matrix
$t=(\alpha_1,\dots,\alpha_d) \in K$ acts on the 3-dimensional space $U_\delta$ as the scalar matrix
$\alpha_i\alpha_j\alpha_k I$.  Hence, if $\delta \neq \delta'$, then $U_\delta$ and $U_{\delta'}$
are non-isomorphic $K$-modules.  Let $M \cong\sym_3$ be the setwise stabiliser of $\delta$. As $M
\le\sym_d\le H$, we may view $U_\delta$ as an $M$-module.  Since $p>3$, $U_\delta$ is a sum of 1-
and 2-dimensional irreducible $M$-submodules. By~\eqref{EA3V} the 1-dimensional submodule is
\[
  A^3 V \cap U_\delta = \langle v_i \wedge v_j \wedge v_k\rangle
  =\langle\; [v_i,v_j]\otimes v_k + [v_j,v_k]\otimes v_i
  + [v_k,v_i]\otimes v_j \;\rangle.
\]
Now $A^3V$ is the direct sum of $\binom{d}{3}$ pairwise non-isomorphic 1-dimensional $K$-submodules,
one for each 3-set $\{i,j,k\}\in\Delta$.  These $K$-submodules are permuted transitively by
$\sym_d$, and so $A^3V$ is an irreducible $H$-module. Now suppose that $N$ is an $H$-submodule where
$A^3V< N \leq W_2$. Choose $x\in N\setminus A^3V$ and write $x= \sum_{\delta \in \Delta} u_\delta$
where $u_\delta\in U_\delta$.  Then there exists $\delta\in\Delta$ for which $u_\delta\not\in
A^3V$. In order to prove that $N=W_2$ it suffices to show that $U_\delta\le N$, as $\sym_d$ is
transitive on $\Delta$.

We claim that $u_\delta\in N$. Assuming the claim is true, then the $M$-submodule $U_\delta\cap N$
satisfies $U_\delta\cap A^3V< U_\delta\cap N\le U_\delta$ and by the above remarks, the only
$M$-submodule of $U_\delta$ properly containing the 1-dimensional submodule $U_\delta\cap A^3V$ is
$U_\delta$ itself. Hence $U_\delta\le N$ and $N=W_2$.

We now prove the claim. Because $\sym_d$ is transitive on $\Delta$, we may assume that $\delta =\{1,2,3\}$. Let
\[
  a:=(-1,1,1,\dots,1),\quad b:=(-1,-1,1,\dots,1),\quad
  c:=(-1,-1,-1,1,\dots,1)
\]
 be elements of $K$.  For $\delta' \in \Delta$, observe that if $1 \in \delta'$, then $u_{\delta'}a
 = -u_{\delta'}$ and if $1\notin \delta'$, then $u_{\delta'}a = u_{\delta'}$.  Thus
 $y:=\frac{1}{2}(x-xa) = \sum_{\delta'\in \Delta,1\in \delta'} u_{\delta'}$, and $y\in N$.  Now
 observe that if $1 \in \delta'$ and $2\notin \delta'$ then $u_{\delta'}b = -u_{\delta'}$, and if
 $\{1,2\}\subset \delta'$ then $u_{\delta'}b = u_{\delta'}$.  Setting $z:=\frac{1}{2}(y+yb)$, we
 have $z=\sum_{\delta' \in \Delta, \{1,2\} \subset \delta'} u_{\delta'}$ and $z\in N$.  Now for all
 $\delta'$ such that $\{1,2\} \subset \delta'$ we have $u_{\delta'}c = u_{\delta'}$ unless
 $\delta'=\{1,2,3\}$. Hence we obtain $u_{\{1,2,3\}} = \frac{1}{2}(z-zc)$.  Thus $u_{\{1,2,3\}} \in
 N \setminus A^3 V$, as desired.  In summary, we have shown that the only $H$-submodule of $W_2$
 properly containing $A^3V$ is $W_2$ itself.  Hence $W_2/A^3 V$ is indeed irreducible as an
 $H$-module. Thus $L^3V=(A^2V\otimes V)/A^3V=X_1\oplus X_2$, where $X_1\cong W_1$ and $X_2\cong
 W_2/A^3 V$ are irreducible.

(ii)~When $d=2$, part~(i) shows that $L^3 V =X_1$ is an irreducible $H$-module, and hence an
 irreducible $\GL(V)$-module.  When $d\ge3$, there is a non-monomial matrix in $\GL(V)$ which maps a
 non-zero element of $X_1$ into $X_2$. This proves that $L^3 V $ is an irreducible $\GL(V)$-module.

(iii) The map $\phi\colon A^2V \otimes V\to L^3V$ given by $\phi([u,v]\otimes w)=[[u,v],w]$ is a
 (well-defined) $\GL(V)$-module homomorphism. Furthermore, it follows from \eqref{EA3V} and the
 Jacobi identity in $L^3V$ that $A^3V\le\ker(\phi)$.  It is clear that $\phi$ is surjective.  We
 observe that
\[
  \dim\left(\frac{A^2V \otimes V}{A^3V}\right)=d\binom{d}{2} - \binom {d}{3}
  =\frac{d^3-d}{3}=\dim(L^3V)
\]
using \eqref{E1}, and hence $\ker(\phi)=A^3V$.

The group algebra $A:=\F\sym_3$ can be written as $A=Ae_1\oplus Ae_2\oplus Ae_3$
where $e_1, e_2, e_3$ are  primitive central orthogonal idempotents where
\[
  e_1=\frac16\sum_{\sigma\in\sym_3}\sigma,\quad e_2=\frac16\sum_{\sigma\in\sym_3}\textup{sign}(\sigma)\sigma,\quad\textup{and}\quad
   e_3=1-e_1-e_2.
\]
Then $T:=T^3V$ equals $TA$, and hence $T=T_1\oplus T_2\oplus T_3$, where
$T_i=Te_i$. However, $T_1=S^3V$ and $T_2=A^3V$, and 
\[
  T=T^2V\otimes V=(S^2V\oplus A^2V)\otimes V
  =(S^2V\otimes V)\oplus(A^2V\otimes V).
\]
By the previous paragraph, $A^2V\otimes V$ has two composition factors: $A^3V$ and $L^3V$. It
follows from the equation $T_1\oplus T_2\oplus T_3=(S^2V\otimes V)\oplus(A^2V\otimes V)$ that $T_3
\cap (A^2V\otimes V)=L^3V$. A similar argument shows that $(V\otimes A^2V)\cap T_3=L^3V$. However,
$A^2V\otimes V\cong V\otimes A^2V$ and $(A^2V\otimes V)\cap (V\otimes A^2V)=A^3V$. Thus
$T_3=L^3V\oplus L^3V$ and so $S^2V\otimes V=S^3V\oplus L^3V$ holds, as desired.
\end{proof}

Finally, we must understand the structure of $L^4V$ when $\dim(V)=2$.

\begin{lemma}
\label{L^4V}
Suppose that $d=2$ and $\Char(\F)\ne 2,3$.  Then $L^4V\cong A^2V\otimes S^2V$ is an irreducible
$\GL(V)$-module.
\end{lemma}

\begin{proof}
Fix a basis $\{e_1,e_2\}$ for $V$.  It is well-known that the left-normed vectors
$[v_1,v_2,v_3,v_4]:=[[[v_1,v_2],v_3],v_4]$ span $L^4V$. Indeed, $\{s_1,s_2,s_3\}$ is a basis for
$L^4V$ where
\[
  s_1=[e_1,e_2,e_1,e_1],\quad
  s_2=[e_1,e_2,e_1,e_2] \quad\textup{and}\quad
  s_3=[e_1,e_2,e_2,e_2].
\]
Note that $s_2=[ e_1,e_2 ,e_2 ,e_1]$.
Define the map
$\phi\colon  A^2V\otimes S^2V \to L^4V$ by:
\[
  \phi([e_1,e_2]\otimes (e_1\odot e_1))=s_1,\quad
  \phi([e_1,e_2]\otimes (e_1\odot e_2))=s_2,\quad
  \phi([e_1,e_2]\otimes (e_2\odot e_2))=s_3.
\]
Since $s_2 = [ e_1,e_2 ,e_2 ,e_1]$, $\phi$ is well-defined. It follows from the linearity of $[,]$
and the universality property of the exterior square, symmetric square and the tensor product, that
$\phi$ is a linear map. Since $\phi$ is surjective, and the dimensions of the respective spaces are
equal, we see that $\phi$ is an isomorphism. Moreover, a direct calculation shows that $\phi$ is a
$\GL(V)$-module isomorphism.  Since $S^2V$ is irreducible as a $\GL(V)$-module and $\dim(A^2V)=1$,
it follows that $L^4V$ is irreducible as a $\GL(V)$-module.
\end{proof}

\section{Aschbacher's Theorem}\label{S4}

An idea pervading Felix Klein's \emph{Erlanger Programm} is that there is a correspondence between
geometry and group theory. A group gives rise to a geometry, and `interesting' subgroups give rise
(via stabilisers) to `interesting' geometric substructures. Our group will be $\GL(d,q)$, where
$q=p^a$, and its `interesting' subgroups will be its \emph{maximal} subgroups $H$. A celebrated
result of Aschbacher relates maximal subgroups of the classical groups to geometry.  For
$\GL(V)\cong \GL(d,q)$, the geometric subgroups fall into eight classes of subgroups which we now
define:
\begin{itemize}
\item[$\C_1$] stabilisers of proper non-zero subspaces of $V$;
\item[$\C_2$] stabilisers of an equidimensional direct sum decomposition $V=V_1\oplus\cdots\oplus V_r$;
\item[$\C_3$] stabilisers of an extension field structure $\F_{q^r}$ where $r$ is prime;
\item[$\C_4$] stabilisers of an unequal dimensional tensor decomposition $V=V_1\otimes V_2$;
\item[$\C_5$] subgroups conjugate (modulo scalars) to a linear group over $\F_{q^{1/r}}$ where $r$ is prime;
\item[$\C_6$] normalisers of an $r$-subgroup of symplectic type where $r\neq p$ is prime;
\item[$\C_7$] stabilisers of an equidimensional tensor product decomposition $V=V_1\otimes\cdots\otimes V_r$;
\item[$\C_8$] stabilisers of non-degenerate forms on $V$.
\end{itemize}

The following statement of Aschbacher's
Theorem follows~\cite[Theorem 1.2.1]{KL}.
An alternative  form of the theorem is given in~\cite[\S3.10.3]{rW}.
 
\begin{theorem}[Aschbacher, \cite{aschbacher}]\label{Asch}
Let $q$ be a power of $p$ and suppose $H\leq\GL(d,q)$ and $\textup{SL}(d,q)\not\leq H$. Then
\begin{enumerate}[{\rm (i)}]
  \item $H$ is contained in a member of (at least one) of the classes $\C_1$~--~$\C_8$, or
  \item $H/\ZZ(H)$ is almost simple and
  $H$ acts absolutely irreducibly on the natural module for $\GL(d,q)$.
\end{enumerate}
\end{theorem}

The subgroups $H$ in Theorem~\ref{Asch} satisfying $H\not\in\C_1\cup\cdots\cup\C_8$ are said to be
of type~$\C_9$.  The size of a maximal subgroup $H$ varies by class: the classes
$\C_1\cup\dots\cup\C_5\cup\C_8$ all contain a `large' subgroup, that is, a subgroup $H$ with $|H|
\geqslant q^{3d+1}$.  On the other hand, for $H\in\C_6\cup\C_7$, we have $|H| < q^{3d+1}$.  To
understand the order of groups in the class $\C_9$, we use the following theorem of Liebeck.

\begin{theorem}[Liebeck~\cite{mL}]\label{T:L}
Let $G_0$ be a simple classical group with natural projective module $V$ of dimension $d$ over
$\mathbb{F}_q$, and let $G$ be a group such that $G_0 \trianglelefteqslant G \leq
\mathrm{Aut}(G_0)$. If $H$ is any maximal subgroup of $G$, then one of the following holds:
\begin{itemize}
\item[(i)] $H$ is a known group (and $H\cap G_0$ has a well-described (projective) action on $V$);
\item[(ii)] $|H| < q^{3d}$.
\end{itemize}
\end{theorem}

For $G_0 \cong \mathrm{PSL}(d,q)$, the remarks in \cite{mL} show that the projective actions of the
groups in part (i) of the theorem above are those of groups from classes $\C_1\cup \dots \cup
\C_5\cup\C_8$. Since every maximal subgroup of $\GL(d,q)$ not containing $\SL(d,q)$ must contain
$\mathrm Z(\GL(d,q))$, we obtain:

\begin{corollary}
\label{Cor:T:L}
Let $G=\GL(d,q)$ and suppose that $H$ is a maximal subgroup of $G$ not containing $\SL(d,q)$. Then
$H\in\C_1\cup \dots \cup \C_5\cup\C_8$, or $|H| < q^{3d+1}$.
\end{corollary}

\section{Representation theory of maximal subgroups on Lie powers}
\label{S5}

We now now assume that $\Char(\F)=p$ is an odd prime and that $\F$ is finite. Recall that $V$ is a
$d$-dimensional vector space over $\F$.  The aim of this section is to determine the reducibility of
$L^2V$, $L^3V$ and $L^4V$ (where necessary) as $H$-modules, for a maximal subgroup $H$ of
$\mathrm{GL}(V)$. In the cases where the modules are reducible, we also aim to determine the
smallest quotient modules.

\subsection{The reducible $\C_1$ case}
\ 

\begin{figure}[H]
\caption{The $\GL(V)_U$  composition factors of $A^2V$  and their respective dimensions.}
\label{F0}
\vskip3mm
\begin{tikzpicture}[xscale=0.6, yscale=0.6]
  \draw[fill] (0,0) circle [radius=0.1];
  \draw[fill] (0,1) circle [radius=0.1];
  \draw[fill] (0,2) circle [radius=0.1];
  \draw[fill] (0,3) circle [radius=0.1];
  \draw[line width=1pt,fill] (0,0) node[left,black] {$\{0\}$} -- (0,1);
  \draw[line width=1pt,fill] (0,1) node[left,black] {$\langle u\wedge u'\mid u, u'\in U\rangle=A^2U$} -- (0,2);
  \draw[line width=1pt,fill] (0,2) node[left,black] {$\langle u\wedge v\mid u\in U, v\in V\rangle=U\wedge V$} --  (0,3) node[left,black] {$A^2V$};
  \draw[thick] (0,2.5) node[right] {$\;A^2(V/U)$};
  \draw[thick] (0,1.5) node[right] {$\;U\otimes (V/U)$};
  \draw[thick] (0,0.5) node[right] {$\;A^2U$};
\end{tikzpicture}
\hspace{12mm}
\begin{tikzpicture}[xscale=0.6, yscale=0.6]
  \draw[fill] (0,0) circle [radius=0.1];
  \draw[fill] (0,1) circle [radius=0.1];
  \draw[fill] (0,2) circle [radius=0.1];
  \draw[fill] (0,3) circle [radius=0.1];
  \draw[line width=1pt,fill] (0,0) node[left,black] {$\{0\}$} -- (0,1);
  \draw[line width=1pt,fill] (0,1) node[left,black] {$A^2U$} -- (0,2);
  \draw[line width=1pt,fill] (0,2) node[left,black] {$U\wedge V$} --  (0,3) node[left,black] {$A^2V$};
  \draw[thick] (0,2.5) node[right] {$\;d_1=\binom{d-r}{2}$};
  \draw[thick] (0,1.5) node[right] {$\;d_2=r(d-r)$};
  \draw[thick] (0,0.5) node[right] {$\;d_3=\binom{r}{2}$};
\end{tikzpicture}
\end{figure}

\begin{lemma}\label{L:C1}
Suppose that  $H=\GL(V)_U\in \mathcal C_1$ is the stabiliser of an
$r$-dimensional subspace $U$ of $V$ where $0<r<d:=\dim(V)$.
\begin{enumerate}[{\rm (i)}]
  \item If $d>2$,  then $L^2 V$ is a reducible $H$-module and the dimension
   of the smallest quotient module is $r$ if $d-r=1$, and
   $\binom{d-r}{2}$ otherwise.
  \item If $p>3$ and $d=2$, then $L^2V$ is an irreducible $H$-module, and
   $L^3 V $ is a uniserial, reducible $2$-dimensional $H$-module.
\end{enumerate}
\end{lemma}

\begin{proof}
(i) We first show that we have a composition series for the $H$-module $A^2V$ as in
  Figure~\ref{F0}. Define $\pi_1 : A^2V \rightarrow A^2(V/U)$ by $\pi_1( v\wedge w) =
  (v+U)\wedge(w+U)$. This map is a surjective $H$-module homomorphism, with kernel
\begin{equation}\label{E:UwedgeV}
  U \wedge V := \langle u \wedge v \mid u\in U, v\in V \rangle.
\end{equation}
Observe that $A^2 U $ is an $H$-invariant subspace of $U\wedge V$.  We claim that $\{0\} \subseteq
A^2 U \subseteq U \wedge V \subseteq A^2V$ is the desired composition series. Note that
$\mathrm{GL}(V/U)$ and $\GL(U)$ act irreducibly on $A^2(V/U)$ and $A^2 U$ respectively.  We
construct an $H$-module isomorphism $ U \otimes (V/U) \cong (U \wedge V)/ A^2 U$ as follows. We
define $\phi \colon U \otimes (V/U) \rightarrow (U \wedge V)/ A^2 U$ to be the linear extension of
the following map:
\[
  u \otimes (v+U) \mapsto u \wedge v + A^2U.
\]
It is straightforward to check that $\phi$ is well-defined, surjective and an $H$-module
homomorphism. Comparing dimensions reveals that $\phi$ is an $H$-module isomorphism and therefore
shows that $(U \wedge V)/ A^2U$ is also irreducible. Hence $A^2 V $ has a composition series as
depicted in Figure~\ref{F0}, where the factors are irreducible or zero.

To prove that $A^2V$ is a uniserial $H$-module, we must show that $\{0\}, A^2U,U\wedge V, A^2V$ are
the only $H$-submodules of $A^2V$ (some may coincide). Using the fact that invertible matrices of
the form $\left(\begin{smallmatrix}I&0\\ *&*\end{smallmatrix}\right)$ lie in $H$, fix $U$
  elementwise and are transitive on $V\setminus U$, it follows that there is no $H$-invariant
  complement to $A^2U$ in $U\wedge V$.  A similar argument shows that there is no $H$-invariant
  complement to $U\wedge V$ in $A^2V$. Thus $A^2V$ is uniserial as claimed and the dimensions of the
  composition factors are as shown in Figure~\ref{F0}.  Note that, since $d>2$, there are at least
  two composition factors, so $A^2V$ is reducible. If $d_1=0$, then $r=d-1$ is the dimension of the
  smallest quotient module.

(ii) Suppose now that $d=2$, $r=1$ and $p>3$.  Then $A^2 V $ is an irreducible 1-dimensional
  $H$-module and $A^3V=\{0\}$.  Since $V$ is a reducible $H$-module, $L^3V\cong A^2V\otimes V$ is a
  uniserial 2-dimensional $H$-module with unique non-trivial submodule $A^2V \otimes U$.
\end{proof}

\subsection{The imprimitive $\C_2$ case}

\begin{lemma}\label{L:C2}
Suppose that $H=\GL(V_1)\wr\sym_r\in \mathcal C_2$ fixes an
equidimensional decomposition
\[
   V = V_1 \oplus \dots \oplus V_r \quad
  \textup{ where $1<r\le d$   and $\Char(\F)=p>2$}.
\]
\begin{enumerate}[{\rm (i)}]
  \item If $1<r <d$, then $L^2V=U_1\oplus U_2$ where $U_1$ and $U_2$ are
   irreducible $H$-modules satisfying 
   $0<  \frac{d}{2} (\frac{d}{r}-1)=\dim(U_1)<\dim(U_2).$
  \item If $p>3$ and $2<r=d$, then $H$ acts irreducibly on $L^2V$, and 
   $L^3V$ is a sum of two irreducible $H$-modules
   of dimensions $2\binom{d}{2}$ and $2\binom{d}{3}$.
  \item If $p>3$ and $2=r=d$, then $H$ acts irreducibly on $L^2V$ and $L^3V$, and
   $L^4V\cong A^2V\otimes S^2V\cong X_1\oplus X_2$ where $\dim(X_1)=2$
   and $\dim(X_2)=1$.
\end{enumerate}
\end{lemma}

\begin{proof}
(i)~Consider the base group $K:=G_1\times\cdots\times G_r$ of $H$ where $G_i=\textup{GL}(V_i)$.  For
  each $i$ we identify $A^2V_i$ with the obvious subspace of $A^2V$. Furthermore, for $i\neq j$ set
  $V_i \wedge V_j:=\langle u \wedge w \mid u\in V_i, w \in V_j \rangle$ mimicking the notation
  in~\eqref{E:UwedgeV}. Then $V_i \wedge V_j = V_j \wedge V_i$, and we note that $V_i \wedge V_j$ is
  isomorphic as a $K$-module to $V_i \otimes V_j$ if $i\ne j$. Hence we have the following
  $K$-module decomposition:
\[
  A^2V=A^2\left(\bigoplus_{i=1}^r V_i\right)=U_1\oplus U_2
  \qquad\textup{where $U_1 \cong \bigoplus_{i=1}^r A^2V_i$ and
    $U_2\cong\bigoplus_{i<j} V_i\otimes V_j$.}
\]

Observe that $A^2V_i$ and $V_i\otimes V_j$ are irreducible $K$-modules.  Thus $A^2V_1,\dots,A^2V_r$
are pairwise non-isomorphic $K$-submodules of $U_1$, and the $V_i\otimes V_j$ with $i<j$ are
pairwise non-isomorphic $K$-submodules of $U_2$ (witnessed by the differing kernels of the action of
$K$).  However, $\sym_r$ permutes these non-isomorphic $K$-modules transitively.  It follows from
Clifford's Theorem \cite[pp. 343--344]{CR} that both $U_1$ and $U_2$ are irreducible $H$-modules.
We have $\dim(U_1)=r \binom{d/r}{2}$, $\dim(U_2 ) = \binom{r}{2} \frac{d^2}{r^2}$, and
$0<\dim(U_1)<\dim(U_2)$.  Hence when $r<d$ we have that $A^2V$ is a reducible $H$-module.

(ii)~Suppose now that $2<r=d$.  By part (i), $U_1 = \{0\}$ and $A^2V=U_2$ is an irreducible
$H$-module.  By Lemma~\ref{L4}(ii), $L^3V=X_1\oplus X_2$ is a sum of irreducible $H$-submodules of
dimensions $2\binom{d}{2}$ and $2\binom{d}{3}$, respectively.

(iii) Finally, consider the case that $2=r=d$. Then $L^2V=A^2V$ is 1-dimensional and $A^3 V
=\{0\}$. Hence $L^3V\cong A^2 V \otimes V$ is the tensor product of an irreducible $H$-module with a
1-dimensional module, and is therefore irreducible.

Restricting the $\GL(V)$-isomorphism $L^4V\cong A^2V\otimes S^2V$ in Lemma~\ref{L^4V}, gives an
$H$-isomorphism.  Now $H$ is generated by matrices of the form
$g=\left(\begin{smallmatrix}0&x\\ y&0\end{smallmatrix}\right)$, where $x$ and $y$ are non-zero, and
  the action of these matrices on $L^4V$ is understood using the map $\phi$ defined in the proof of
  Lemma~\ref{L^4V}. It follows that the following is an $H$-module decomposition:
\[
  A^2V\otimes S^2V\cong \langle v_1\wedge v_2\otimes v_1\odot v_1,
  v_1\wedge v_2\otimes v_2\odot v_2\rangle \oplus \langle v_1\wedge
  v_2\otimes v_1\odot v_2\rangle.
\]
These 2- and 1-dimensional $H$-submodules are irreducible, as desired.
\end{proof}

\subsection{The extension field $\C_3$ case}

We assume that $\F=\F_q$ is finite, $\Char(\F)=p$ and let $\E=\F_{q^r}$ with $r$ a prime.  In this
case, $\Gamma{\mathrm L}(1,\E/\F)$ is a maximal subgroup of $\GL(r,p)$ by \cite[Theorem 1.2.1]{KL}.
The $r$th cyclotomic polynomial $\Phi_r(t)$ factors over $\F_q$ as a product of equal-degree
irreducibles by~\cite[Theorem~2.47(ii),~p.\,61]{LN}.  This common degree divides $r-1$.

\begin{lemma}
\label{L:GammaL1}
Let $\E=\F_{q^r}$ and $\F=\F_q$ where $r$ is a prime and $q$ a power of the prime $p$.  Let $V$ be
an irreducible $\Gamma\mathrm L(1,\E/\F)$-module over~$\F$. Then $\dim(V)$ equals $r$, or
divides~$r-1$.  In particular, the maximum dimension of an irreducible $\Gamma\mathrm
L(1,\E/\F)$-module over $\F$ is $r$.
\end{lemma}

\begin{proof}
Observe that $\Gamma\mathrm L(1,\E/\F)$ is isomorphic to the metacyclic group
\[
  H=\langle\phi,\mu\mid \phi^r=\mu^{q^r-1}=1,\mu^\phi=\mu^q\rangle.
\]
Consider $V^\E= V \otimes_\F \E$ as an $\E M$-module where $M=\langle\mu\rangle$.  As $|M|=q^r-1$ is
coprime to~$p$, it follows that $V^\E$ is a completely reducible $M$-module by Maschke's
Theorem. Let $W$ be an irreducible $\E M$-submodule of~$V^\E$.  Thus $\dim_\E(W)=1$, as $\E$ is a
splitting field for $M$. Hence $\mu$ acts as a non-zero scalar, $\lambda(\mu)\in \E$~say on $W$.

\textsc{Case: $\lambda(\mu)=\lambda(\mu^q)$.}  Since $\lambda(\mu) = \lambda(\mu)^q$, we have
$\lambda(\mu) \in \F$. Then $\langle\mu\rangle$ acts on $V$ as the matrix $\lambda(\mu) I$. It
follows that $V$ is an irreducible $\F H$-module if and only if $V$ is an irreducible
$\langle\phi\rangle$-module. Thus, by the remarks preceding this lemma, $\dim(V)$ divides~$r-1$.

\textsc{Case: $\lambda(\mu)\neq \lambda(\mu^q)$.}  Let $U = \bigoplus_{i=0}^{r-1} W^{\phi^i}$. Note
that $W$ is not isomorphic to $ W^\phi$ as an $\E M$-module by assumption. Hence $U$ is the sum of
pairwise non-isomorphic $\E M$-submodules, which are permuted transitively by $H$. Thus $U$ is an
irreducible $\E H$-module. Note also that $U$ is a summand of $V^\E $. By \cite[26.6(1)]{ASCH} we
have that $V$ is a summand of the restriction $U_\F$, of $U$ to $\F$. By \cite[VII~Theorem
  1.16(e)]{HB2}, $U_\F$ is a direct sum of isomorphic modules, each of dimension $\dim_\E
(U)=r$. Hence $\dim_\F(V)=r$.
\end{proof}

The computational algebra systems~\cite{BCP} and~\cite{GAP} were used to investigate the submodule
structure of Lie powers for $\C_3$ groups~$H$.  The first $n$ for which $L^nV$ was $H$-reducible
turned out to be completely reducible. From the data we collected, we could guess, but not prove,
the dimension of the smallest quotient $H$-module of $L^nV$. Thus we suspect that the three
inequalities that appear in Table~\ref{Tab2} (in the $\C_3$ rows) are in fact equalities.

\begin{lemma}\label{L:C3}
Suppose that $H=\GL(d/r,\F_{q^r})\rtimes\mathrm{Gal}(\F_{q^r}/\F_q)\in\C_3$ is a subgroup of
$\GL(V)$ where $V=(\F_q)^d$, and $r$ is a prime, and suppose $\Char(\F_q)=p>2$.
\begin{enumerate}[{\rm (i)}]
  \item If $1<r < d $ then $H$ acts reducibly on $L^2V$, preserving a
   quotient of dimension~$\binom{d/r}{2}r$.
  \item If $3<r=d$ then $H$ acts reducibly on $L^2V$, with a minimal
    quotient of dimension~$d$.
  \item If $3=r=d$ and $p>3$, then $H$ acts irreducibly on $L^2V$, and reducibly
    on $L^3V$.
  \item If $2=r=d$ and $p>3$, then $H$ acts irreducibly on $L^2V$ and $L^3V$,
    and reducibly on~$L^4V$.
\end{enumerate}
\end{lemma}

\begin{proof}
(i) As above write $\E =\F_{q^r}$ and $\F=\F_q$.  We think of $H$ as acting semilinearly on $V'=\E
  ^{d/r}$, and view $V$ as $(V')_\F$, i.e., $V'$ with scalars restricted to $\F$.  Thus $\dim_\E
  (V')=d/r$ and $\dim_\F(V)=d$. Similarly, let $T' = A^2V'$, and let $T=(T')_\F$.  Since $d/r>1$, we
  have $\dim_\F(T)=r\dim_\E (T')=r\binom{d/r}{2}>0$.  We construct a surjective $\F H$-module
  homomorphism $\eta\colon A^2V\rightarrow T$.  Certainly $\ker(\eta)$ is a \emph{proper} submodule
  of $A^2V$ because $\dim(T)>0$, and $\ker(\eta)$ is non-zero because
\[
  \dim(\ker(\eta))=\dim(A^2V)-\dim(T)
  =\binom{d}{2}-r\binom{d/r}{2}=\frac{d(d-d/r)}{2}>0.
\]
Fix a basis $\alpha_1,\dots,\alpha_r$ for $\E $ over $\F$ and a basis $v_1,\dots,v_{d/r}$ for
$V'$. Then $V$ has a basis
\[
  \{ \alpha_i v_j \mid 1 \leqslant i \leqslant r, 1 \le j \le d/r \}
  \]
and $T$ has a basis
\[
  \{ \alpha_i v_{j} \wedge v_k \mid 1 \leqslant i \leqslant r, 1 \le j < k \le d/r \}.
  \]
Furthermore, $A^2V$ has a basis consisting of vectors of the form $\alpha_i v_k \wedge \alpha_j
v_\ell$. As $\alpha_i\alpha_j\in \E $, we may write $\alpha_i\alpha _j = \sum_{s=1}^r\lambda_s
\alpha _s$ where $\lambda_s\in \F$. Define $\eta\colon A^2V\to T$ by
\[
  \eta(\alpha_i v_k \wedge \alpha _j v_\ell) = (\alpha _i \alpha _j)
  v_k \wedge v_\ell = \left(\sum_{s=1}^r \lambda_s \alpha _s\right)
  v_k \wedge v_\ell = \sum_{s=1}^r \lambda_s (\alpha _s v_k \wedge
  v_\ell).
\]
Certainly $\eta$ is a $\GL(V')$-homomorphism, and $\eta(\beta v_k\wedge\gamma v_\ell)=\beta\gamma
v_k\wedge v_\ell$ for all $\beta,\gamma\in \E $. As $\theta\in\mathrm{Gal}(\E /\F)$ maps $\alpha_i
v_k$ to $(\alpha_i^\theta) v_k$, we see that $\eta((\alpha_i v_k \wedge \alpha _j v_\ell)^\theta)$
equals
\[
  \eta(\alpha_i^\theta v_k \wedge \alpha _j^\theta v_\ell) = (\alpha
  _i^\theta \alpha _j^\theta) v_k \wedge v_\ell = \sum_{s=1}^r \lambda_s
  (\alpha _s^\theta v_k \wedge v_\ell) =\eta(\alpha_i v_k \wedge \alpha
  _j v_\ell)^\theta.
\]
Hence $\eta$ is an $H$-homomorphism as desired. Since $\eta$ is a surjective $\F H$-homomorphism,
and $0<\dim(\ker(\eta))<\dim(A^2V)$, $H$ acts reducibly on $A^2V$.  As $\GL(V')$ acts irreducibly on
$A^2V'$, it follows that $H$ acts irreducibly on $T$.

(ii) Suppose that $d=r$ is prime and $r>3$. Then $H\cong \mathrm C_{q^d-1} \rtimes \mathrm C_d$. We
adopt the notation in the proof of Lemma~\ref{L:GammaL1} and write $H=\langle\phi,\mu\mid
\phi^d=\mu^{q^d-1}=1,\mu^\phi=\mu^q\rangle$.  Let $e_0,e_1,\dots,e_{d-1}$ be a basis for $V$ over
$\F=\F_q$.  Let $A$ and $C$ be the $d\times d$ matrices over $\F$ corresponding to the action of
$\phi$ and $\mu$ on $V$. Now let $\E =\F_{q^d}$ and set $V^\E =V\otimes_\F \E $.  Since~$C$ is
irreducible over $\F$, its characteristic polynomial has distinct roots $\zeta, \zeta^q,\dots,
\zeta^{q^{d-1}}$ in~$\E $. Thus $C$ is conjugate in $\GL(d,\E )$ to the diagonal matrix $C^\E
:=\mathrm{diag}(\zeta, \zeta^q,\dots, \zeta^{q^{d-1}})$.  Let $A^\E $ be the matrix with $e_iA^\E
=e_{i+1}$ where the subscripts are read modulo~$d$. Then $A^\E $ satisfies $(C^\E )^{A^\E }=(C^\E
)^q$, and it follows that there exists a matrix in $\GL(d,\E )$ that conjugates $A$ to~$A^\E $ and
$C$ to~$C^\E $. The matrices $A$, $C$ in $\GL(V)$ induce matrices $a$, $c$ in $\GL(A^2V)$ and $A^\E
$, $C^\E $ in $\GL(V^\E )$ induce matrices $a^\E $, $c^\E $ in $\GL(A^2V^\E )$. The induced matrices
$a, c\in \GL(A^2 V)$ and $a^\E , c^\E \in\GL(A^2 (V^\E ))$ are (simultaneously) conjugate in
$\GL(A^2 (V^\E ))$.

The action of $a^\E $ and $c^\E $ relative to the basis $e_i\wedge e_j$, $0\le i<j<d$, for $A^2V^\E$
is given by $e_i\wedge e_ja^\E =e_{i+1}\wedge e_{j+1}$ and $e_i\wedge e_jc^\E
=\zeta^{q^i+q^j}e_i\wedge e_j$.  We show that a typical eigenvalue $\xi_{i,j}=\zeta^{q^i+q^j}$ of
$c^\E $ does not lie in $\F$.  Indeed, suppose that $\xi_{i,j} \in \F$, then
$\xi_{i,j}^{q^{j-i}}=\xi_{i,j}$ and $\zeta^{q^j+q^{2j-i}}=\zeta^{q^i+q^j}$. Since
$\zeta^{q^{2(j-i)}}=1=\zeta^{q^d-1}$, and $q^d-1$ is coprime to $q^{2(j-i)}$, it follows that
$\zeta$ has order~1, a contradiction. As $\xi_{i,j}$ is an eigenvalue of $c$, it follows that $c$
does not fix an $\F$-subspace of dimension less than $d$.  The $d$-dimensional $\E $-subspace
$U=\langle e_i\wedge e_{i+1}\mid 0\le i<d\rangle$, is invariant under $a^\E $ and $c^\E $. The
restrictions of $a^\E $ and $c^\E $ to $U$ have matrices $a^\E _U=A$ and $c^\E
_U=\mathrm{diag}(\xi_{0,1},\xi_{0,1}^q,\dots,\xi_{0,1}^{q^{d-1}})$, respectively.  The subgroup
$\langle a^\E _U,c^\E _U\rangle$ is irreducible by Clifford's Theorem \cite[pp. 343--344]{CR}.  A
simple calculation shows that the character values of the monomial group $\langle a^\E _U,c^\E
_U\rangle$ lie in $\F$, so by a theorem of Brauer~\cite[VII~Theorem 1.16(e)]{HB2}, the subgroup
$\langle a^\E _U,c^\E _U\rangle$ of $\GL(d,\E )$ is conjugate to an irreducible subgroup of
$\GL(d,\F)$. In summary, we have proved that every non-zero $H$-submodule of $A^2V$ has
$\F$-dimension at least $d$, and one has dimension precisely~$d$. As $H$ can be shown to act
completely reducibly on $A^2V$, it follows that the smallest dimensional proper quotient module of
$A^2V$ has dimension~$d$.

(iii) Suppose that $d=r=3$. The argument in part~(ii) shows that $H$ preserves an irreducible
3-dimensional subspace of $A^2V=L^2V$.  Thus $H$ acts irreducibly on $L^2V$. By~\eqref{E1},
$\dim(L^3V)=(3^3-3)/3=8$ so by Lemma~\ref{L:GammaL1}, $H$ acts reducibly on $L^3V$ preserving a
submodule of codimension at most~$3$.

(iv) Suppose that $d=r=2$. Then $H$ acts irreducibly on the 1-dimensional space $L^2 V$, and on the
2-dimensional space $L^3 V\cong A^2V\otimes V$.  Finally, $H$ acts reducibly on $L^4V$ by
Lemma~\ref{L:GammaL1} as $\dim(L^4V)=(2^4-2^2)/4=3$, and $H$ preserves a submodule of codimension at
most~$2$.
\end{proof}

\subsection{The tensor reducible $\C_4$ case}

\begin{lemma}\label{L:C4}
  Suppose that $H=\GL(V_1)\circ\GL(V_2)\in \mathcal C_4$ where $2\le\dim(V_1)<\dim(V_2)$ and
  $\Char(\F)\neq 2$. Then $L^2(V_1\otimes V_2)=U_1\oplus U_2$ where $U_1\cong A^2V_1\otimes S^2V_2$
  and $U_2\cong S^2V_1\otimes A^2V_2$ are irreducible $H$-modules satisfying
  $0<\dim(U_1)<\dim(U_2)<\dim( L^2(V_1\otimes V_2))$.
\end{lemma}

\begin{proof}
Let $H=\GL(V_1)\circ\GL(V_2)$ preserve the decomposition $V = V_1 \otimes V_2$ where
$2\le\dim(V_1)<\dim (V_2)$. By~\eqref{E:A+S}, we have the following $H$-module isomorphisms
\begin{align*}
  T^2V&=(V_1\otimes V_2)\otimes(V_1\otimes V_2)\\ &\cong(V_1\otimes
  V_1)\otimes(V_2\otimes V_2)\\ &\cong \left(S^2V_1\oplus
  A^2V_1\right)\otimes \left(S^2V_2\oplus A^2V_2\right)\\ &\cong
  \left(S^2V_1\otimes S^2V_2\;\oplus\; A^2V_1\otimes A^2V_2\right)
  \;\oplus\; \left(S^2V_1\otimes A^2V_2\;\oplus\; A^2V_1\otimes
  S^2V_2\right)\\ &\cong S^2V\oplus A^2V.
\end{align*}
Equating symmetric and anti-symmetric parts gives the following $H$-module isomorphisms:
\begin{align*}
  S^2 V &\cong S^2V_1 \otimes S^2V_2 \;\oplus\; A^2V_1 \otimes A^2V_2,
    \quad\textup{and}\\
  A^2 V &\cong S^2V_1 \otimes A^2V_2 \;\oplus\; A^2V_1 \otimes S^2V_2.
\end{align*}
In particular, we see that $A^2V\cong L^2V$ is reducible as an $H$-module. Since $S^2V_i$ and
$A^2V_i$ are irreducible $\GL(V_i)$-submodules (for $i=1,2$), it follows that $S^2V_1\otimes A^2V_2$
and $A^2V_1\otimes S^2V_2$ are irreducible modules for $\GL(V_1)\times\GL(V_2)$ and hence for
$H=\GL(V_1)\circ\GL(V_2)$.  Since $2\le d_1<d_2$ where $d_1=\dim(V_1)$ and $d_2=\dim(V_2)$, it is
easy to see that
$0<\binom{d_1}{2}\binom{d_2+1}{2}<\binom{d_1+1}{2}\binom{d_2}{2}<\binom{d_1d_2}{2}$, and hence
$0<\dim(U_1)<\dim(U_2)<\dim(L^2(V_1\otimes V_2))$.
\end{proof}

\subsection{The tensor induced case $\C_7$}
%\section{Proof for the \texorpdfstring{$\C_7$}{C7} case}\label{S:C7}

The classes $\C_i$ considered so far all contain `large' maximal subgroups of $\GL(d,p)$, i.e., ones
with $|H| \geqslant p^{3d+1}$.  By contrast, none of the $\C_7$ subgroups $H$ are large in this
sense; indeed Corollary~\ref{Cor:T:L} shows that $|H| < p^{3d+1}$. Intuitively, the smaller $|H|$ is
compared to $|\GL(d,p)|$ the less likely it is that modules with dimensions much larger than~$d$
remain irreducible, when restricted to~$H$.  Thus one might expect that our desired $p$-group~$G$
(with $A(G)=H$) has small nilpotency class, and that it is not too hard to construct. The first
expectation is true, but not the second, as the small dimensional modules such as $L^2V$ and $L^3V$
turn out to be hard to handle.

\begin{theorem}\label{T:C7}
Let $H=\GL(V_1)\wr \sym_r\le\GL(V)$ preserve the tensor decomposition $V=V_1 \otimes \dots \otimes
V_r$, so $H\in\C_7$.  Suppose that $p:=\Char(\F) > 2$, $r\ge2$, and
$t:=\dim(V_1)=\cdots=\dim(V_r)\ge2$.

\begin{enumerate}[{\rm (i)}]
  \item If $p>2$ and $r>2$, then $L^2V$ is reducible and the smallest quotient module of $L^2V$ has
    dimension $\binom{t}{2}^r$ if $r$ is odd, and $r\binom{t}{2}^{r-1}\binom{t+1}{2}$ if $r$ is
    even.
  \item If $p>3$ and $r=2$, then $L^2V$ is an irreducible $H$-module,
    and $L^3V$ is a reducible $H$-module.  The smallest dimension of a
    quotient module of $L^3V$ is $4$ if $t=2$, and
    $(t+1)t^2(t-1)^2(t-2)/9$ if~$t>2$.
\end{enumerate}
\end{theorem}

\begin{proof}
As $H\in \mathcal C_7$, we have
$H=\GL(V_1)\wr\sym_r\le\GL(V_1^{\otimes r})$ where $t\ge2$
and $r\ge2$.

(i) Suppose now that $V=V_1\otimes\cdots\otimes V_r$
where $r\ge2$ and $p>2$. Rearranging tensor factors, and using~\eqref{E:A+S}
shows that
\[
T^2V=T^2V_1\otimes\cdots\otimes T^2V_r =(A^2V_1\oplus
S^2V_1) \otimes \cdots \otimes (A^2V_r\oplus S^2V_r).
\]
Expanding gives $2^r$ summands. We show that collecting these summands into $\sym_r$-orbits gives
$T^2V=\bigoplus_{k=0}^r U_k$ where the $U_k$ are pairwise non-isomorphic irreducible $H$-submodules
satisfying
\[
  A^2V=\bigoplus_\textup{$k$ odd} U_k,\quad
  S^2V=\bigoplus_\textup{$k$ even} U_k,\quad\textup{and}\quad
  \dim(U_k)=\binom{r}{k}\binom{t}{2}^k\binom{t+1}{2}^{r-k}.
\]
We identify the $2^r$ summands with the elements of the vector space $C=(\F_2)^r$.  The orbits of
$\sym_r$ on the vectors of $C$ are $C_0,\dots,C_r$ where $C_k$ comprises the $\binom{r}{k}$ vectors
with precisely $k$ ones. Define
\[
  U_k=\bigoplus_{(\eps_1,\dots,\eps_r)\in C_k}
  X^{\eps_1}(V_1)\otimes\cdots\otimes
  X^{\eps_r}(V_r)\quad\textup{where}\quad
  X^{\eps_i}(V_j)=\begin{cases}A^2V_j&\textup{if
    $\eps_i=1$,}\\S^2V_j&\textup{if $\eps_i=0$.}\end{cases}
\]
The summands of $U_k$ are pairwise non-isomorphic irreducible modules for the base group
$\GL(V_1)\times\cdots\times\GL(V_r)$ of $H$, so by Clifford's Theorem \cite[pp. 343--344]{CR}, $U_k$
is an irreducible $H$-submodule. The formula for $\dim(U_k)$ is now clear as
$\dim(A^2V_i)=\binom{t}{2}$ and $\dim(S^2V_j)=\binom{t+1}{2}$ by \eqref{E:AS}.

The number of irreducible $H$-submodules $U_k$ of $A^2V$ is the number of odd $k$ satisfying $0\le
k\le r$, namely $\lceil r/2\rceil$.  Hence $A^2V$ is reducible precisely when $r>2$.  Suppose that
$k_0$ is odd and $\dim(U_{k_0})\le\dim(U_k)$ for all odd $k$ satisfying $0\le k\le r$. Observe first
that $r-k<k$ implies that $\dim(U_{r-k})>\dim(U_k)$ so we may assume $r/2\le k_0\le r$. Second, note
that if $k, \ell$ are odd and $r/2\le\ell<k$, then it follows that $\dim(U_\ell)>\dim(U_k)$ because
$\binom{r}{\ell}>\binom{r}{k}$. Hence $k_0=r$ when~$r$ is odd, and $k_0=r-1$ when~$r$ is even. This
proves part~(i).

(ii) Suppose now that $p>3$, $r=2$, and $V=V_1\otimes V_2$. By part~(i), $L^2V$ is irreducible. We
use Lemma~\ref{L4} to investigate the $K$-module structure of $A^2V\otimes V$ where
$K=\GL(V_1)\times\GL(V_2)$ is normal in $H$ of index~2. It follows from part~(i) that we have the
following $K$-module decomposition: $A^2V= (A^2V_1\boxtimes S^2V_2) \oplus (S^2V_1\boxtimes A^2V_2)$
where $\boxtimes$ denotes `outer tensor product' for $K$.  Consider the following $K$-module
decomposition:
\begin{align*}
  A^2V\otimes V &\cong \left((A^2V_1\boxtimes S^2V_2)\oplus
  (S^2V_1\boxtimes A^2V_2)\right) \otimes (V_1\boxtimes V_2)\\ &\cong
  (A^2V_1\otimes V_1)\boxtimes (S^2V_2\otimes V_2) \oplus
  (S^2V_1\otimes V_1)\boxtimes (A^2V_2\otimes V_2).
\end{align*}
Lemma~\ref{L4}(ii) gives $A^2V_i\otimes V_i\cong L^3V_i\oplus A^3V_i$ and
$S^2V_i\otimes V_i\cong S^3V_i\oplus L^3V_i$, so
\begin{align*}
  A^2V\otimes V &\cong (L^3V_1\oplus A^3V_1)\boxtimes
  (S^3V_2\oplus L^3V_2) \oplus (S^3V_1\oplus L^3V_1)\boxtimes (L^3V_2\oplus A^3V_2)\\ &\cong
  (B_1\oplus C_1)\boxtimes(A_2\oplus B_2)\oplus (A_1\oplus
  B_1)\boxtimes(B_2\oplus C_2)
\end{align*}
where $A_i=S^3V_i$, $B_i=L^3V_i$, and $C_i=A^3V_i$.
Expanding shows that $A^2V\otimes V$ is
a sum of 8 irreducible $K$-modules as follows:
\begin{equation}\label{E:PQRS}
  A^2V\otimes V\cong P\oplus Q\oplus R\oplus S
\end{equation}
where
\begin{align*}
  P&=A_1\boxtimes B_2\oplus B_1\boxtimes A_2,\qquad &&Q=A_1\boxtimes
  C_2\oplus C_1\boxtimes A_2,\\ R&=B_1\boxtimes C_2\oplus C_1\boxtimes
  B_2, &&S=B_1\boxtimes B_2\oplus B_1\boxtimes B_2.
\end{align*}

By Clifford's Theorem \cite[pp. 343--344]{CR}, $P$, $Q$ and $R$ are pairwise non-isomorphic
irreducible $H$-modules, whilst $S$ is the sum of two irreducible $H$-modules, $S_1$ and $S_2$ say,
each isomorphic to $B_1 \otimes B_2$. Using Lemma~\ref{L4}(iii), we reconcile the $H$-decompositions
\[
  A^2V\otimes V=L^3V\oplus A^3V\quad\textup{and}\quad
  A^2V\otimes V=P\oplus Q\oplus R\oplus S_1\oplus S_2.
\]
\begin{table}[!ht]
\renewcommand{\arraystretch}{1.8} 
\caption{Dimensions of irreducible $H$-submodules of $A^2V\otimes V$.}
\label{Tab3}\kern-2mm\begin{center}
\begin{tabular}{|c||c|c|c|c|c||c|c|c|c|} \hline
$U$&$P$&$Q$&$R$&$S_1$&$S_2$&$a=\dim(A_i)$&$b=\dim(B_i)$&$c=\dim(C_i)$&$d$\\ \hline
  $\dim(U)$&$2ab$&$2ac$&$2bc$&$b^2$&$b^2$&$\frac{(t+2)(t+1)t}{6}$&$\frac{(t+1)t(t-1)}{3}$&$\frac{t(t-1)(t-2)}{6}$&$t^2$\\ \hline
\end{tabular}
\end{center}
\end{table}

The dimensions of the modules $P$, $Q$, $R$, $S_1$ and $S_2$ are displayed in
Table~\ref{Tab3}. Since $L^3V$ is a completely reducible $H$-module, there exist $p,q,r,s_1,s_2 \in
\{0,1\}$ such that
\[
  \dim L^3V = \frac{t^6-t^2}{3} = p \dim P + q\dim Q + r \dim R
  +s_1\dim S_1 + s_2 \dim S_2.
\]
The above gives rise to 32 polynomial equations in $t$. If $t\neq 4$, then the only solutions are
$(p,q,r,s_1,s_2) = (1,0,1,1,0)$ or $(p,q,r,s_1,s_2) = (1,0,1,0,1)$. If $t=4$, then there are two
additional possibilities since $\dim R = \dim Q$, namely that $(p,q,r,s_1,s_2) = (1,1,0,0,1)$ or
$(p,q,r,s_1,s_2) = (1,1,0,1,0)$.  Renumbering if necessary, assume that $S_1\le L^3V$ and thus
$S_2\le A^3V$.  Hence, if $t \neq 4$ we obtain $L^3V \cong P\oplus R \oplus S_1$. When $t= 4$ the
additional possibility that $L^3V\cong P \oplus Q \oplus S_1$ arises. As $L^3V$ is completely
reducible, the smallest non-zero quotient $H$-module is isomorphic to the smallest irreducible
$H$-submodule of $L^3V$.  If $t=2$ then $c=0$ and $L^3V \cong P \oplus S_1$ and the minimal
dimension of an $H$-submodule of $L^3V$ is $4$. If $t>2$ then $c > 0$ and the dimensions of the
minimal $H$-submodules of $L^3V$ are $2ab$, $2bc$ and~$b^2$. Since $a>c$ and $b>2c$, the smallest
dimension of a minimal submodule of $L^3V$ in this case is $2bc=(t+1)t^2(t-1)^2(t-2)/9$.
\end{proof}

\subsection{The $\C_8$ case, classical groups in natural action}

As our primary interest is in the field $\F_p$, we do not consider the unitary groups here. The
following remark elucidates the symplectic case in Lemma~\ref{L:C8}(i).

\begin{remark}\label{R:extra}
The extraspecial group~$G$ of order $p^{1+2m}$ with exponent~$p>2$ has
a pc-presentation
\begin{equation}\label{Epc}
  G=\left\langle g_1,\dots,g_{2m+1}\mid g_1^p=\dots=g_{2m+1}^p=1,
  \ g_{2k}^{g_{2k-1}}=g_{2k}g_{2m+1},\ 1\le k\le m\right\rangle
\end{equation}
where $g_j^{g_i}=g_j$ for $1\le i<j\le 2m+1$ and $(i,j)\not\in\{(2k-1,2k)\mid 1\le k \le m\}$.
Using collection, we can symbolically multiply
\begin{align*}
  \big(g_1^{x_1}g_2^{y_1}\cdots
  g_{2m-1}^{x_m}g_{2m}^{y_m}g_{2m+1}^{z}\big)
  \big(g_1^{x'_1}g_2^{y'_1}&\cdots
  g_{2m-1}^{x'_m}g_{2m}^{y'_m}g_{2m+1}^{z'}\big)\\ &=g_1^{x_1+x'_1}g_2^{y_1+y'_1}\cdots
  g_{2m-1}^{x_m+x'_m}g_{2m}^{y_m+y'_m} g_{2m+1}^{z+z'+\sum_{k=1}^m
    x_ky'_k}.
\end{align*}
However, writing $v_1=(x_1,y_1,\dots,x_m,y_m)$ and $v'_1=(x'_1,y'_1,\dots,x'_m,y'_m)$, we have a
more symmetric multiplication rule on pairs in $\F_p^{2m} \times \F_p$:
\[
  (v_1,v_2)(v'_1,v'_2)=(v_1+v'_1,v_2+v'_2+\beta(v_1,v_1'))
\]
where $\beta(v_1,v'_1)=\sum_{k=1}^m (x_ky'_k-x'_ky_k) \pmod p$. This rule is a `quotient' of the Lie
$2$-tuple rule in Example~\ref{R}, and it helps to show that the conformal symplectic group
$\textup{CSp}(\beta)$ is a subgroup of~$\Aut(G)$.  If $g\in\textup{CSp}(\beta)$ satisfies
$\beta(v_1g,v'_1g)=\beta(v_1,v'_1)\delta_g$ where $\delta_g\in\F$ is non-zero, then the map
$(v_1,v_2)\alpha_g=(v_1g,v_2\delta_g)$ lies in $\Aut(G)$, and $g\mapsto\alpha_g$ is a monomorphism
$\textup{CSp}(\beta)\to\Aut(G)$. This proves that $\Aut(G)$ \emph{splits} over $\textup{Inn}(G)$,
c.f.~\cite[Theorem~1(a)]{Winter}.
\end{remark}

\begin{lemma}
\label{L:C8}
Suppose that $H\in \mathcal C_8$ is the stabiliser of a non-degenerate form on $V=(\F_q)^d$, where
$q$ is an odd prime power and $d>2$.
\begin{enumerate}[{\rm (i)}]
  \item If $H$ preserves an alternating form, then $H$ acts reducibly on $L^2V$, and the smallest
    dimension of a quotient module is~$1$.
  \item If $H$ preserves a quadratic form, then $H$ acts irreducibly on $L^2V$, and reducibly on
    $L^3V$. Moreover, the smallest dimension of a quotient module of $L^3V$ is $d$ or $1$.
\end{enumerate}
\end{lemma}

\begin{proof}
(i)~Suppose that $H=\mathrm{CSp}(\beta)$ is the conformal symplectic group preserving the
  alternating form $\beta\colon V\times V\to\F_q$ up to scalar multiples. Recall that
  $\mathrm{CSp}(\beta)/\mathrm{Sp}(\beta)\cong\F_q^\times\cong\mathrm{C}_{q-1}$.  The linear map
  $\pi\colon L^2V \rightarrow \F_q$ satisfying $\pi([v,w] )=\beta(v,w)$ is well-defined precisely
  because $\beta$ is alternating.  Moreover, since $\beta$ is an $H$-invariant form we have that
  $\pi$ is an $H$-module homomorphism, and $\mathrm{CSp}(\beta)$ acts non-trivially on $\F_q$ with
  kernel $\mathrm{Sp}(\beta)$.  Clearly $\pi$ is onto, therefore $\dim(L^2V/\ker(\pi))=1$.  As
  $\dim(L^2V)=\binom{d}{2}>1$ for $d>2$, we see that $L^2V$ is reducible as claimed.

(ii) Suppose that $H$ preserves the symmetric form $\beta\colon V\times V\to\F_q$ up to non-zero
  scalar multiples.  Since $p$ is odd, $H$ acts irreducibly on $A^2V$, see \cite[Table 1]{mL}.
  Define $\pi\colon T^3V \rightarrow V \otimes \F_q$ by $\pi(u \otimes v \otimes w)=u\otimes
  \beta(v,w)$.  Since $H$ preserves $\beta$ up to scalars, we see that $\pi$ is an $H$-module
  homomorphism. Moreover, since
\[
  u\wedge v \wedge w = u \otimes v \otimes w - u \otimes w \otimes v +
  v \otimes w \otimes u - v \otimes u \otimes w + w \otimes u \otimes
  v - w \otimes v \otimes u
\]
we have
\[
  \pi (u\wedge v \wedge w)= u\otimes (\beta(v,w)-\beta(w,v) ) + v
  \otimes (\beta(w,u) - \beta(u,w)) + w \otimes
  (\beta(u,v)-\beta(v,u)).
\]
Thus $\pi(A^3V)=\{0\}$ since $\beta$ is symmetric. Now choose vectors $u$, $v$ and $w$ of $V$ so
that $u\otimes v \otimes w$ is a fundamental tensor and such that $f(u,w) = 0 $ and $\beta(v,w) \neq
0$ (such a choice is always possible since $\beta$ is non-degenerate). Then $x:=u \otimes v \otimes
w - v\otimes u \otimes w \in A^2V \otimes V$ and $\pi(x) = u \otimes \beta(v,w) \neq 0$. Hence
\[
  A^3V \le\ker(\pi) \cap (A^2V \otimes V) < A^2V \otimes V
\]
 and the quotient $ (A^2V \otimes V) / (\ker(\pi) \cap (A^2V \otimes V) )$ is isomorphic to a
 submodule of $V\otimes \F_q$. Since the latter is an irreducible $H$-module, we have that the
 smallest quotient module of $L^3V$ has dimension $d$ or $1$.
\end{proof}

\begin{remark}
We do not consider the case when $H$ is a maximal subgroup of $\GL(d,p)$ containing $\SL(d,p)$. In
this case the irreducible $\GL(V)$-submodules of $L^nV$ with $p>n$, are likely to restrict to
irreducible $\SL(V)$-modules.  In the case $d=2$ excluded in Lemma~\ref{L:C8}, $H$ contains
$\textup{Sp}(2,p)=\SL(2,p)$.
\end{remark}

\section{Proof of the main theorem}
\label{S7}

\begin{table}[!ht]
\renewcommand{\arraystretch}{1.8} % so lines in tables are not crowded
\caption{The exponent-$p$ groups $G$ of class $n$ in
  Theorem~\ref{T:main} for different Aschbacher classes $\C_i$ where
  $|G|=p^m$ and $m=\sum_{i=1}^{n-1}f(d,i)+\dim(G_{n-1})$.}
\label{Tab2}\kern-2mm
\begin{center}\def\kk{\kern-1.5pt}
\begin{tabular}{|c|c|c|c|c|c|c|c|} \hline
$\C_i$&$V\kern-1pt=\kern-1ptG_0/G_1$&$H$&\textup{conditions}&$n$&$p\kern-3pt\ge$&$\dim(G_{n-1})$&$G_{n-1}$\\ \hline
$\C_1$&$0<U<V$&$\GL(V)_U$&$1\kk<\kk r\kk<\kk d-1$&2&3&$\binom{d-r}{2}$&$A^2(V/U)$\\
 &  &$r:=\dim(U)$&$1\kk<\kk r\kk=\kk  d-1$&2&3&$r$&$U\otimes(V/U)$\\
 & &  &$(d,r)=(2,1)$&3&5&$1$&$A^2V\otimes (V/U)$\\ \hline
$\C_2$&$\bigoplus_{i=1}^rV_i$&$\GL(V_1)\wr\sym_r$&$1<r<d$&2&3&$\binom{d/r}{2}r$&$U_1$\\
 & &$d=r\dim(V_1)$&$4<r=d$&3&5&$d(d-1)$&$W_1$\\
 & & &$3,4=r=d$&3&5&$2\binom{d}{3}$&$W_2/A^3V$\\
 & & &$2=r=d$&4&5&$1$&\textup{Lemma \ref{L:C2}}\\ \hline
$\C_3$&$(\F_{p^r})^{d/r}$&$\Gamma\mathrm L(d/r,\F_{p^r})$&$1<r<d$&2&3&$\le\binom{d/r}{2}r$&\textup{Lemma \ref{L:C3}(i)}\\
 & & &$3<r=d$&2&3&$d$&\textup{Lemma \ref{L:C3}(ii)}\\
 & & &$3=r=d$&3&5&$\le 3$&\textup{Lemma \ref{L:C3}(iii)}\\
 & & &$2=r=d$&4&5&$\le2$&\textup{Lemma \ref{L:C3}(iv)}\\ \hline
$\C_4$&$V_1\otimes V_2$&$\GL(V_1)\circ\GL(V_2)$&$1<d_1<d_2$&2&3&$\binom{d_1}{2}\binom{d_2+1}{2}$&$A^2V_1\otimes S^2V_2$ \\
 & &$d_i:= \dim(V_i)$&$d=d_1d_2$&&&&\textup{Lemma \ref{L:C4}}\\ \hline
$\C_7$&$\bigotimes_{i=1}^r  V_1$&$\GL(V_1)\wr\sym_r$&$2<r$&2&3&\textup{\ref{T:C7}(i)}&$U_{2\lfloor(r-1)/2\rfloor}$\\
 & &$d=\dim(V_1)^r$&$2=r$&3&5&\textup{\ref{T:C7}(ii)}&$R$\textup{ if $t>2$}\\ \hline
$\C_8$&  &$\textrm{CSp}(\beta)$&$2<d$&2&3&$1$&\textup{det}\\
 &  &$\textrm{GO}(\beta)$ &$2<d$&3&5&$1,d$&\textup{Lemma \ref{L:C8}} \\ \hline
\end{tabular}
\end{center}
\end{table}

In this section we complete the proof of Theorem~\ref{T:main}. In fact, we prove a stronger theorem
from which Theorem~\ref{T:main} follows, after an application of Corollary~\ref{Cor:T:L}.

\begin{theorem}
\label{T:other theorem}
Let $p \geqslant 5$ be a prime, and let $d \geqslant 2$ be an integer. Suppose that $H$ is a maximal
subgroup of $\GL(d,p)$ with $\SL(d,p)\not\le H$ and that $H$ lies in one of the Aschbacher classes
$\C_1 \cup \dots \cup \C_5 \cup \C_7 \cup C_8$.  Then there exists a $d$-generator $p$-group~$G$ of
exponent~$p$, class at most~$4$, order at most $p^{\frac{d^4}{2}}$ and $A(G)=H$.  The nilpotency
class, order and structure of~$G$ is given in Table~{\rm\ref{Tab2}}.

\end{theorem}
\begin{proof}
Let $H$ be as in the statement of the theorem and let $V=\F_p^d$. Note that $H$ cannot be in class
$\C_5$ and cannot be in class $\C_8$ preserving a unitary form. We seek a $d$-generator $p$-group
$G$ of exponent~$p$ and minimal class such that $A(G)=H$. Now $\GL(V)$ (and hence $H$) acts on the
sections of the lower exponent-$p$ central series of the $d$-generator Burnside group~$B=B(d,p)$.
By Lemmas~\ref{L:C1}, \ref{L:C2}, \ref{L:C3}, \ref{L:C4}, \ref{L:C8} and Theorem~\ref{T:C7} there
exists an $n\le 4$ such that $H$ acts irreducibly on $L^1 V ,\dots, L^{n-1} V $ (with the exception
that if $H$ is of class $\mathcal C_1$ then $H$ is reducible on $L^1V$), and there is a maximal
$H$-submodule, say $M/B_n$, of $B_{n-1}/B_n\cong L^nV$ which is not $\GL(V)$-invariant. Set
$G:=B/M$. We claim that $G$ is the desired $p$-group.

Since $B_n < M<B_{n-1}$ is $H$-invariant, $G$ is a proper quotient of the finite group
$\Gamma_n(V)=\Gamma(d,p,n)$. In particular, $G$ has class $n$. Now
$H\le\textup{N}_{\GL(V)}(M/B_n)\le\GL(V)$ and since $H$ is maximal in $\GL(V)$, our choice of $M$
gives $\textup{N}_{\GL(V)}(M/B_n) =H $.  Hence Theorem ~\ref{T:trick} gives
$A(G)=\textup{N}_{\GL(V)}(M/B_n)=H$.

It remains to bound $|G|$.  By construction, $G$ is a quotient of $\Gamma(d,p,n)$, and the order of
the latter group is given in Theorem~\ref{T1}. From this it easily follows that $|G| \le
p^{\frac{d^4}{2}}$ as claimed.
\end{proof}

\begin{remark}
\label{rem:category}
For a given $H \le \GL(d,p)$, we let $\mathcal G(H)$ be the category of all finite $d$-generator
$p$-groups $P$ with $A(P)=H$. Then the group $G$ appearing in Theorem~\ref{T:other theorem} is the
minimal element of $\mathcal G(H)$ with respect to order and nilpotency class. In fact, if $H\in
\C_1 \cup \C_2 \cup\C_4 \cup \C_7$ or $H$ is a $\C_8$ subgroup preserving a symplectic form, we have
also found the groups in $\mathcal G(H)$ of minimal order.
\end{remark}

\begin{remark}
Let $H$ be the $\mathcal C_1$ maximal subgroup $\GL(V)_U$ which fixes a proper non-zero subspace $U$
of $V$. Let $r=\dim(U)$ and let $P=(C_p)^r \times (C_{p^2})^{d-r}$. Then $P$ is abelian and of
exponent~$p^2$, and it is easy to check that $A(P)=H$.  The group $P$ has smaller order than the
corresponding group listed in Table~\ref{Tab2}, but the exponent is $p^2$ rather than $p$.
\end{remark}

\section{Some open questions}
\label{S8}

Aschbacher's Theorem~\cite[Theorem 1.2.1]{KL} and the results of Sections~\ref{S3}, \ref{S4},
\ref{S5} work over an arbitrary finite field~$\F_q$. There is no definition of `$q$-groups' where
$q=p^f$ and $f>1$. However, taking a group $\Gamma_n(\F_q^d)$ defined in Construction~\ref{const:lie
  tuple} results in a group that has a Frattini quotient isomorphic to $\F_q^d$.  Unfortunately,
these groups are not relatively free since they are $df$-generator groups and the lower central
series of $\Gamma_n(\F_q^d)$ is not the same as that of $\Gamma_n(\F_p^{df})$.

How must our results be modified when $p=2$?  How large must the nilpotency class of~$G$ be in the
cases $\C_6$ and $\C_9$ which contain no `large' subgroups?  How do the multiplication
rules~\eqref{E2.1}--\eqref{E2.3} for the universal groups $\Gamma_n(\F^d)$ generalise for $n>4$? To
what extent can collection in groups of exponent~$p$ given by pc-presentations be replaced by
\emph{symbolic} computations in Lie $n$-tuple groups? (This type of question is explored
in~\cite{LGS}, for example.)

Suppose that $H$ is a maximal subgroup of $\GL(V)$ and the irreducible $\GL(V)$-submodules of $L^1 V
,\dots,L^{n-1} V $ restrict to irreducible $H$-submodules, and $n$ is maximal with this
property. Our examples lead us to ask: Is $L^nV$, viewed as an $H$-module, always either completely
reducible or uniserial?

\medskip

\noindent\textsc{Acknowledgements.}  The first author acknowledges the support of the Australian
Research Council Future Fellowship FT120100036, the second and fourth authors acknowledge the
support of the Australian Research Council Discovery Grant DP140100416 and the third author
acknowledges the ARC grants DP120100446 and DE160100081.  We would like to thank Kay Magaard, Csaba
Schneider and Oihana Garaialde Ocan\~{a} for helpful comments, and Saul Freedman for carefully
reading the manuscript.

\end{document}